\documentclass[11pt,reqno]{amsart}
\usepackage{amsmath,amssymb,amsthm}
\usepackage{color}
\usepackage[tmargin=1in,bmargin=1in,rmargin=1in,lmargin=1in]{geometry}
\usepackage[breaklinks=true]{hyperref}
\usepackage{enumitem}
\usepackage{mathrsfs}

\theoremstyle{plain}
\newtheorem{theorem}{Theorem}[section]
\newtheorem{corollary}[theorem]{Corollary}
\newtheorem{lemma}[theorem]{Lemma}
\newtheorem{proposition}[theorem]{Proposition}

\newtheorem{remark}[theorem]{Remark}

\theoremstyle{definition}

\numberwithin{equation}{section}
\numberwithin{table}{section}
\allowdisplaybreaks

\newcommand{\tl}{\widetilde{L}}
\newcommand{\tp}{\widetilde{P}}

\newcommand{\tm}{\widetilde{M}}

\newcommand{\rg}{\mathcal{R}(\widetilde{G}(F))}

\newcommand\bc{\mathbb C}
\newcommand\bz{\mathbb Z}

\newcommand\gf{\widetilde{G}(F)}
\newcommand\hf{\widetilde{G}(F')}
\newcommand\go{\widetilde{G}(\mathscr{O})}

\title{A Hecke algebra isomorphism over close local fields for metaplectic groups}
\author{Ritabrata Das}

\address{Department of Mathematics, Indian Institute of Science, Bangalore 560012, India;
PhD student
}
\email{ritabratadas@iisc.ac.in}


\begin{document}
\begin{abstract}   
We establish the Kazhdan isomorphism for the \(n\)-fold metaplectic cover of the reductive group \(\mathrm{SL}_2\) over two sufficiently close non-archimedean local fields \(F\) and \(F'\), both of which have residue characteristic coprime to \(n\) and contain all distinct \(n\)th roots of unity.
\end{abstract}

\maketitle
\section{Introduction}

Let $F$ be a local non-archimedean field with the ring of integers $\mathscr{O} $ and the maximal ideal $\mathscr{P} \subset \mathscr{O}$. Let $F$ contain the $n$-distinct $n$th roots of unity $\mu_n(F)$, and $n$ is coprime to the residue characteristic. The group $G(F)= \mathrm{SL_2}(F)$ has a unique nontrivial central extension $\widetilde{G}(F)$ of degree $n$:
\begin{align*}
    1 \longrightarrow \mu_n(F) \overset{i} \longrightarrow \widetilde{G}(F) \overset{p} \longrightarrow G(F) \longrightarrow 1.
\end{align*}
The group $Mp_2(F):=\gf$ is the $n$-fold metaplectic cover of $G(F)$. In this paper, we establish a Hecke algebra isomorphism over close local fields for $\gf$.

 For an object $X$ associated with the field $F$, we will use the notation $X'$ to denote the corresponding object over $F'$. In \cite{Kazhdan} Kazhdan  proved 
that for $G$, being a split, connected reductive group defined over $\mathbb{Z}$ and given $l \ge 1$, there exists $m \ge l$ such that if $F$ and $F'$ are 
$m$-close, then there is an isomorphism of Hecke algebras
\[
\mathrm{Kaz}_l : \mathcal{H}(G(F), K_l) \longrightarrow 
\mathcal{H}(G(F'), K'_l),
\]
where $K_l$ is the $l$-th usual congruence subgroup of $G(\mathscr{O})$. Hence, when the fields $F$ 
and $F'$ are sufficiently close, we have a bijection
\[
\begin{aligned}
&\{\text{Irreducible admissible } \mathbb{C}\text{-representations } (\pi,V) \text{ of } G(F) \text{ such that } \pi^{K_l} \neq 0\} \\
&\longleftrightarrow \{\text{Irreducible admissible } \mathbb{C}\text{-representations } (\pi',V') \text{ of } G(F') \text{ such that } \pi'^{K'_l} \neq 0\}.
\end{aligned}
\]

Ganapathy in \cite{hecke} extended this result for general connected reductive groups. Kazhdan's philosophy has been useful in studying the local Langlands correspondence for certain groups over local fields of positive characteristic (\cite{Jacquet}, \cite{Ganapathy-Varma1},\cite{Ganapathy-Varma2}).

There are three crucial steps in the proof of the Kazhdan isomorphism of Hecke algebras for reductive groups:

\begin{enumerate}
    \item[(1)] The group $G(F)$ admits a Cartan decomposition. 
    \item[(2)]  If the fields $F$ and $F'$ are $l$-close, then we obtain a natural isomorphism 
    \begin{align*}
        G(\mathscr{O})/K_l \cong G(\mathscr{O'})/K_l'.
    \end{align*}

    \item[(3)] The Hecke algebra $\mathcal{H}(G(F), K_l)$ is finitely presented.
\end{enumerate}

 We write $\mathscr{O}$ for the ring of integers of $F$. Let $\delta$ denote the Kubota cocycle for our metaplectic cover $\gf$ and let $S$ denote the corresponding splitting of $G(\mathscr{O})$ (see Section~2). For any subgroup $H$ of $G(\mathscr{O})$, we denote its image under the splitting by $H^*:=S(H)$.

 Fixing a faithful character $\varepsilon:\mu_n(F)\hookrightarrow \mathbb C^{\times}$, we define the genuine Hecke algebra with respect to the congruence subgroup $K_l^*$ for the covering group $\widetilde{G}(F)$ as follows 
\begin{align*}
\mathcal{H}_{\varepsilon} (\widetilde{G}(F),K_l^* )
:=\left\{
f :\widetilde{G}(F) \longrightarrow \mathbb C \,:\,
\begin{aligned}
&\bullet\ \text{$f$ is locally constant and compactly supported},\\
&\bullet\ f(\zeta g)=\varepsilon(\zeta)f(g),
\quad \forall \zeta \in \mu_n(F),\ g\in \widetilde{G}(F),\\
&\bullet\ f(k_1gk_2)=f(g),
\quad \forall k_1,k_2\in K_l^*,\ g \in \widetilde{G}(F).
\end{aligned}
\right\}.
\end{align*}
The algebra structure is given by the convolution of functions:
\begin{align*}
    (f * h)(x)=\int_{\gf} f(y) h(y^{-1}x) dy ,         
\end{align*}
where $dy$ is the Haar measure of the group $\gf$, which is normalized such that the measure of $\mu_n(F)\times K_l^*$ is $1$.

In this paper, we establish an isomorphism between the genuine Hecke algebras over close local fields. Our proof relies on establishing analogs of the three crucial steps (1), (2), and (3) of Kazhdan's proof for reductive groups. However, we do encounter several technical difficulties along the way.

\qquad Because the unipotent subgroups of $G(F)$ split canonically in $\gf$ (see \cite[Prop. 4.1]{principal}, \cite[Appendix 1]{MoeglinWaldspurger}), we can fix preferred liftings of the elements in the diagonal torus. These are obtained via the canonical liftings of the root subgroups (see Section~\ref{lift of torus}). 

We derive the Cartan decomposition for the $n$-fold covering group of $\mathrm{SL_2}$ in Section 3. This result is implicit in Theorem 9.2 of \cite{principal}. Peskin explicitly worked out the special case for the $2$-fold metaplectic cover in \cite[Lemma 2.6]{peskin}.
 
A main difficulty in proving the isomorphism between the genuine Hecke algebras lies in understanding the behavior of Kubota's cocycle on the $K_l^{*}$-double cosets in $\widetilde{G}(\mathscr{O})$. We carefully analyze the Kubota cocycle in Section 5 and provide technical lemmas in Section 7 to show that the natural vector space isomorphism between the genuine Hecke algebras respects convolution when the fields are sufficiently close.

Finally, we show that the genuine Hecke algebra is finitely presented, which is a key step in proving the Hecke algebra isomorphism. For reductive groups, this property is proved in \cite[\S 2]{BDKV}, where the authors remarked that the proof should also extend to covers of reductive groups. The proof relies heavily on the Bernstein decomposition, which was established for covering groups in \cite{KaplanDani}. We then follow the framework laid out for reductive groups in \cite{BDKV} to complete the proof.

Our main result is stated as follows:
\begin{theorem}\label{intro theorem} ( For a precise version, see Theorem~\ref{isomorphism})
    Let $l$ be an integer such that $l\geq 1$. Then, there exists an integer $m\in \mathbb{Z}^+$ with $m\ge l$ such that, for any non-archimedean local field $F'$ which is $m$-close to $F$, we have the following algebra isomorphism 
\begin{align*}
\Phi_l:\mathcal{H}_\varepsilon(\widetilde{G}(F),K_l^*) \longrightarrow 
\mathcal{H}_{\varepsilon'}(\widetilde{G}(F'),K_l'^*).
\end{align*}
\end{theorem}

By combining this genuine Hecke algebra isomorphism with the equivalence of categories established in Section 8, we obtain the following corollary:
\begin{corollary} 
Let $F$ and $F'$ be sufficiently close local fields, and let $l$ be as in Theorem~\ref{intro theorem}. Then, we have the following bijection: 
\[
\begin{aligned}
&\{\text{ Irreducible genuine $\bc$-representations } (\pi,V) \text{ of } \gf\text{ such that  } \pi^{K_l^*}\neq 0\} \\
&\longleftrightarrow \{\text{ Irreducible genuine $\bc$- representations }  (\pi',V') \text{ of } \hf\text{ such that  } \pi'^{K_l'^*}\neq 0 \} .
\end{aligned}
\]
\end{corollary}

This result allows us to study the representations of the $n$-fold metaplectic cover of $\mathrm{SL_2}$ over close local fields.

\section{Notations}
\subsection{}\label{hilbert}Let us fix a uniformizer $t$ of $F$, and the valuation $v$ of $F$ is normalized so that $v(t)=1$. The Hilbert symbol of degree $n$ is denoted by $(\cdot,\cdot)_F$. Let $p$ be the residue characteristic of $F$ and let $q$ be the cardinality of the residue field.  Since $(n,p)=1$, the Hilbert symbol for any $a,b\in F^\times$ can be calculated using the following formula, given in \cite[Prop. 3.4]{Neukirch} 
\begin{align}\label{hilbert symbol formula}
    (a,b)_F=\omega\Big( (-1)^{v(a)v(b)} \frac{b^{v(a)}}{a^{v(b)}}\Big)^{\frac{q-1}{n}},
\end{align}
where for $x\in\mathscr{O}^\times$ with decomposition $x=\omega(x)x'\in \mu_{q-1}(F)\times (1+\mathscr{P)}$; $\omega(x)$ denotes the unique $\mu_{q-1}(F)$ part of $x$.
\subsection{}\label{kubota}In \cite{Kubota}, Kubota provided an explicit construction of the metaplectic group as follows. As a set,
\begin{align*}
   \gf=G(F)\times\mu_n(F),
\end{align*}
and the group operation is defined by  $(g,a)(h,b)=(gh, \delta(g,h)ab)$, where $\delta\in H^2(G(F),\mu_n(F))$ is Kubota's cocycle.

The cocycle is explicitly given by the formula
\begin{equation}\label{kubota cocycle formula}
\delta(g,h) = \left( \frac{X(gh)}{X(g)},\frac{X(gh)}{X(h)} \right)_F,
\end{equation}
where $(\cdot,\cdot)_F$ denotes the $n$-th Hilbert symbol of $F$, and the function $X$ is defined on matrices by
\[
X\begin{pmatrix} a & b \\ c & d \end{pmatrix} = 
\begin{cases} 
c & \text{if } c \neq 0, \\ 
d & \text{if } c = 0. 
\end{cases}
\]
By \cite[\S3, Thm.~2]{Kubotaauto}, the Kubota cocycle $\delta$ allows for a splitting of $G(\mathscr{O})$. This splitting is a homomorphism $S:G(\mathscr{O})\longrightarrow \widetilde{G}(F)$ defined by $g\longmapsto (g,\theta(g))$, where 
\[
\theta\begin{pmatrix} a & b \\ c & d \end{pmatrix} = 
\begin{cases} 
1 & \text{if } c=0 \text{ or } c\in\mathscr{O}^\times, \\ 
(c,d)_F & \text{otherwise}. 
\end{cases}
\]
 That is, $\theta(gh)=\theta(g)\theta(h)\delta(g,h)$. For any subgroup $H \subset G(\mathscr{O})$, we set $H^* := S(H)$. 
For $h \in H$, we denote the element $S(h)=(h,\theta(h)) \in H^*$ by $h^*$.

 \qquad   Consider the congruence subgroup $K_m=Ker (G(\mathscr{O}) \longrightarrow G(\mathscr{O}/\mathscr{P}^m))$ in $G(\mathscr{O})$. For $m\geq 1$, if $k=\begin{pmatrix}
a & b \\
c & d
\end{pmatrix}\in K_m$, then $d\in 1+\mathscr{P}^m$ and by the formula \eqref{hilbert symbol formula}, it is clear that $\theta(k)=1$. Hence, $K_m^*=\{(k,1):k\in K_m\}$ for $m\geq 1$.

\subsection{}\label{unipotent splitting}The unipotent subgroups of $G(F)$ split canonically ( see \cite[Prop. 4.1]{principal}, \cite[Appendix 1]{MoeglinWaldspurger}) in $\gf$. Let $U_+$ and $U_-$ denote the upper triangular and lower triangular unipotent subgroups of $G(F)$, respectively. 

The canonical splittings for $U_+$ and $U_-$ with respect to Kubota's cocycle are as follows
\begin{align}\label{splitting unipotent upper}
   U_+ \rightarrow \gf, \hspace{.5cm}
u_+(x):=\begin{pmatrix}
     1 & x\\
     0& 1
\end{pmatrix}
\mapsto  \tilde{u}_+(x):=(u_+(x),1)
\end{align}
and,
\begin{align}\label{splitting unipotent lower}
   U_- \rightarrow \gf, \hspace{.5cm}
u_-(x)=\begin{pmatrix}
     1 & 0\\
     x& 1
\end{pmatrix}
\mapsto \tilde{u}_-(x):=(u_-(x),1) .
\end{align}
\qquad It is straightforward to verify that \eqref{splitting unipotent upper} defines a splitting for the unipotent upper triangular subgroup $U_+$. Thus, it remains only to show that \eqref{splitting unipotent lower} yields a splitting for the unipotent lower triangular subgroup $U_-$. To prove this, it suffices to show that $\delta(u_-(x), u_-(y)) = 1$ for all $x,y\in F$. If either $x=0$ or $y=0$, the result is immediate. Hence, we may assume that both $x$ and $y$ are nonzero.
We have
\begin{equation*}
u_-(x) u_-(y) = \begin{pmatrix} 1 & 0 \\ x+y & 1 \end{pmatrix}.
\end{equation*}

First, suppose that $x+y=0$. Then, by the formula \eqref{kubota cocycle formula}
\[
\delta(u_-(x), u_-(y)) = \left(\frac{1}{x}, \frac{1}{y}\right)_F=\left(\frac{1}{x}, -\frac{1}{x}\right)_F = 1.
\]
The last equality follows from \cite[Prop.~3.2(v), §3, Ch.~V]{Neukirch}.

\qquad Now assume that $x+y \neq 0$. In this case,
\[
\delta(u_-(x), u_-(y)) 
= \left(\frac{x+y}{x}, \frac{x+y}{y}\right)_F
= \left(\frac{x}{x+y}, \frac{y}{x+y}\right)_F
= \left(\frac{x}{x+y}, 1 - \frac{x}{x+y}\right)_F
= 1.
\]
Here, we have used the fact that $(a,b)_F=(a^{-1},b^{-1})_F$, which is clear from the formula \eqref{hilbert symbol formula} and the last equality again follows from \cite[Prop.~3.2(v), §3, Ch.~V]{Neukirch}.
This completes the proof.

\subsection{}\label{lift of torus} For $x\in F^*$, we define the elements
\begin{align*}
    &w(x)=u_+(x)u_-(-x^{-1})u_+(x),\\
    &\lambda_0(x)=w(x)w(-1)=\begin{pmatrix}
        x&0\\
        0& x^{-1}
    \end{pmatrix}.
\end{align*}
Using the canonical liftings of $U_+$ and $U_-$ from Section~\ref{unipotent splitting}, we choose the preferred liftings of $w(x)$ and $h(x)$ as follows:
\begin{align*}
    &\tilde{w}(x)=\tilde{u}_+(x)\tilde{u}_-(-x^{-1})\tilde{u}_+(x)=(w(x),1),\\
    &\tilde{\lambda}_0(x)=\tilde{w}(x)\tilde{w}(-1)=(\lambda_0(x),(-1,x)_F).
\end{align*}
\begin{proof}
By definition, the element $\tilde{w}(x)$ is given by
\begin{align*}
   \tilde{w}(x) = \big( \omega(x), \delta(\tilde{u}_+(x), \tilde{u}_-(-x^{-1})) \, \delta(\tilde{u}_+(x)\tilde{u}_-(-x^{-1}), \tilde{u}_+(x)) \big). 
\end{align*}

We first evaluate the individual cocycle terms. For the first factor, an immediate calculation yields
\[
\delta\big(\tilde{u}_+(x), \tilde{u}_-(-x^{-1})\big) = (-x^{-1}, 1)_F = 1.
\]
For the second factor, expressing the arguments explicitly as matrices, we have
\[
\delta\big(\tilde{u}_+(x)\tilde{u}_-(-x^{-1}), \tilde{u}_+(x)\big) = \delta \left( \begin{pmatrix} 0 & x \\ -x^{-1} & 1 \end{pmatrix}, \begin{pmatrix} 1 & x \\ 0 & 1 \end{pmatrix} \right) = (1, -x^{-1})_F = 1.
\]
Combining these evaluations, we obtain $\tilde{w}(x) = (\omega(x), 1)$.

Similarly, for the element $\tilde{\lambda}_0(x)$, we have
\[
\tilde{\lambda}_0(x) = \big(\lambda_0(x), \delta(\omega(x), \omega(-1))\big).
\]
Evaluating this final cocycle term yields
\[
\delta(\omega(x), \omega(-1)) = \delta \left( \begin{pmatrix} 0 & x \\ -x^{-1} & 0 \end{pmatrix}, \begin{pmatrix} 0 & -1 \\ 1 & 0 \end{pmatrix} \right) = (-1, x^{-1})_F = (-1, x)_F,
\]
where we used the fact that $(-1, x^{-1})_F = (-1, x)_F$. It follows that 
\[
\tilde{\lambda}_0(x) = \big(\lambda_0(x), (-1, x)_F\big),
\]
which completes the proof.
\end{proof}

\subsection{}\label{torus splitting} For the uniformizer $t$ of $F$, the formula for the Hilbert symbol \eqref{hilbert symbol formula} gives the following (\cite[Sec.~ 2.1.5]{peskin}): 

\[
\tilde{\lambda}_0(t)^k =(\lambda_0(t)^k,\beta_k),
\]
where
\[
\beta_k =
\begin{cases}
(-1,t)_F &\text{ if } k \equiv 1 \text{ or } 2\pmod{4}, \\
1 &\text{ if } k \equiv 0 \text{ or } 3\pmod{4}.
\end{cases}
\]
Also, $\beta_k \beta_{-k}=(-1,t)_F^k$.

\begin{proof}
From the formula of the Hilbert symbol in \eqref{hilbert symbol formula}, we have the identity
\begin{equation}\label{fact fromhilbert symbol}
    (-1,t)_F = (t,t)_F = (t,t^{-1})_F.
\end{equation}
Using this relation, we compute $\beta_k$ as follows:
\begin{align*}
    \beta_k &= (-1,t)_F^k \, (t,t)_F \, (t,t)_F^2 \cdots (t,t)_F^{k-1} \\
    &= (-1,t)_F^k \, (-1,t)_F^{\frac{k(k-1)}{2}} \\
    &= (-1,t)_F^{\frac{k(k+1)}{2}}.
\end{align*}
The parity of the exponent $\frac{k(k+1)}{2}$ depends on the congruence class of $k$ modulo $4$. Specifically, we have
\[
\frac{k(k+1)}{2} \text{ is } 
\begin{cases} 
    \text{even} & \text{if } k \equiv 0 \text{ or } 3 \pmod{4}, \\ 
    \text{odd}  & \text{if } k \equiv 1 \text{ or } 2 \pmod{4}. 
\end{cases}
\]
Furthermore, computing the product $\beta_k \beta_{-k}$ yields
\[
\beta_k \beta_{-k} = (-1,t)_F^{\frac{k(k+1)}{2} + \frac{-k(-k+1)}{2}} = (-1,t)_F^{k^2} = (-1,t)_F^k.
\]

\end{proof}

\subsection{} \label{torus normalize unipotent}

 Combining \eqref{splitting unipotent upper} with the results from Section~\ref{torus splitting}, we have the following:
\begin{align*}
    \tilde{\lambda}_0(t)^{k}\tilde{u}_{+}(x)\tilde{\lambda}_0(t)^{-k}
    &= \left(u_{+}(t^{2k}x),\, (t,t^{-1})_F^{k^2}\beta_k\beta_{-k}\right) \\
    &= \left(u_{+}(t^{2k}x),\, (t,t^{-1})_F^{k^2} (-1,t)_F^k\right) \\
    &= \left(u_{+}(t^{2k}x),\, 1\right) \\
    &= \tilde{u}_+(t^{2k}x),
\end{align*}
where we used the identity $\beta_k\beta_{-k} = (-1,t)_F^k$ from Section~\ref{torus splitting} along with the relation $(t,t^{-1})_F = (-1,t)_F$ from \eqref{fact fromhilbert symbol}. 

Similarly, for the lower unipotent subgroup, a direct computation shows that
\begin{align*}
    \tilde{\lambda}_0(t)^{k}\tilde{u}_{-}(x)\tilde{\lambda}_0(t)^{-k}
    &= \left(u_{-}(t^{-2k}x),\,(x,t^{-k})_F (t^{-k},t^{-k}x)_F \beta_k\beta_{-k}\right) \\
    &= \left(u_{-}(t^{-2k}x),\,(x,t^{-k})_F (t^{-k},t^{-k})_F (t^{-k},x)_F (-1,t)_F^k\right) \\
    &= \left(u_{-}(t^{-2k}x),\,(x,t^{-k})_F (t^{-k},x)_F (t,t)_F^{k^2} (-1,t)_F^k\right) \\
    &= \left(u_{-}(t^{-2k}x),\,1\right) \\
    &= \tilde{u}_-(t^{-2k}x). 
\end{align*}
Here, the simplification relies on the facts that $(a,b)_F (b,a)_F = 1$ for all $a,b \in F^{\times}$ \cite[Prop.~3.2(iv), §3, Ch.~V]{Neukirch}, and that $(t,t)_F = (-1,t)_F$ via \eqref{fact fromhilbert symbol}.
\section{Cartan Decomposition} 
Let  $T \subset G(F)$  be the maximal split torus consisting of diagonal matrices, and let $B \subset G(F)$ be the Borel subgroup of upper triangular matrices containing $T$. Denote the characters by $X^{*}(T)$ and the cocharacters by $X_{*}(T)$. Then, $X^*(T)\cong \bz\chi_{_0}$, where $\chi_{_0}(\begin{pmatrix}
    x&0\\
    0& x^{-1}
\end{pmatrix})=x$ and $X_*(T)\cong \bz \lambda_0$, where $\lambda_0(x)=\begin{pmatrix}
    x&0\\
    0& x^{-1}
\end{pmatrix}$. The set of roots associated with $(T, B)$ is $\phi=\{\pm 2\chi\}$. Since $\langle \chi_{_0},\lambda_0\rangle=1$, the set of dominant coweights is given by
\begin{align*}
    \Lambda^{+} &=\{\lambda \in X_{*}(T) : \langle\alpha,\lambda \rangle \geq 0 \hspace{.2cm}\forall \alpha \in \phi^{+}\}\\
    &=\{k\cdot \lambda_0 : k\in \bz_{\ge 0}\}.
\end{align*}
The group $G(F)$ has the following Cartan decomposition:
\begin{align*}
    G(F)
    &=\bigsqcup_{\lambda \in \Lambda^{+}} G(\mathscr{O}) \lambda(t^{-1}) G(\mathscr{O)}.\\
    &=\bigsqcup_{k\geq 0} G(\mathscr{O}) \lambda_0(t)^{-k} G(\mathscr{O)}.
\end{align*}

From now on, we are going to denote $\tilde{\lambda}_0(t)^{-k}$ by $\tilde{t}_k$.
\begin{remark}\label{before Cartan}
    $\widetilde{G}(F)=\underset{k\in \bz_{\geq 0}}{\bigsqcup}\widetilde{G}(\mathscr{O}) \tilde{t}_k \widetilde{G}(\mathscr{O})$ as observed in \cite[Section 2.3]{peskin}.
\end{remark}
Proof. Let $(g,\zeta)\in \widetilde{G}(F)$. The Cartan Decomposition of $G(F)$ yields $g=g_1 t_k g_2$ for some $g_1,g_2\in G(\mathscr{O})$ and $k\in \bz_{\geq 0}$. Say,
\begin{align*}
    (g_1,1)\tilde{t}_k(g_2,1)=(g,\eta), \hspace{.2cm} \textit{ where} \hspace{.2cm}\eta\in\mu_n(F).
\end{align*}
Then, $(g,\zeta)=(g,\eta)(1,\eta^{-1}\zeta)=(g_1,1)\tilde{t}_k(g_2,\eta^{-1}\zeta)\in \widetilde{G}(\mathscr{O}) \tilde{t}_k \widetilde{G}(\mathscr{O})$.\\
Suppose there exists an element $(g,\zeta)$ such that $(g,\zeta)\in\widetilde{G}(\mathscr{O}) \tilde{t}_k \widetilde{G}(\mathscr{O})\cap\widetilde{G}(\mathscr{O}) \tilde{t}_r \widetilde{G}(\mathscr{O})$, for some $k\neq r\in \bz_{\geq 0}$. That will give us that $g\in G(\mathscr{O})t_kG(\mathscr{O})\cap G(\mathscr{O})t_r G(\mathscr{O})$, a contradiction. Hence, the union is disjoint.

\begin{proposition}\label{cartan} 
   $\widetilde{G}(F)$ has the Cartan decomposition
    \begin{align*}
      \widetilde{G}(F) = \bigsqcup_{\substack{k\geq0 \\ \zeta \in \mu_n(F)}}  G(\mathscr O)^{*}\tilde{t}_k \zeta G(\mathscr O)^{*}.
    \end{align*}   
\end{proposition}
\begin{proof}The disjointness of the union is implicit in \cite[Theorem 9.2]{principal}, and we adopt a similar approach here.

Any element in the double coset $\widetilde{G}(\mathscr{O}) \tilde{t}_k \widetilde{G}(\mathscr{O})$ can be expressed as $g_1^*\tilde{t}_k\zeta g_2^*$ for some $\zeta\in\mu_n(F)$ and $g_i^*=(g_i,\theta(g_i))\in G(\mathscr{O})^*$. This yields the decomposition
\begin{align*}
    \widetilde{G}(\mathscr{O}) \tilde{t}_k \widetilde{G}(\mathscr{O}) = \bigcup_{\zeta \in \mu_n(F)} G(\mathscr{O})^{*}\tilde{t}_k\zeta G(\mathscr{O})^{*}.
\end{align*}
Consequently, by Remark~\ref{before Cartan}, we obtain
\begin{align*}
    \widetilde{G}(F) = \bigsqcup_{k\geq0 } \left( \bigcup_{\zeta\in\mu_n(F)} G(\mathscr{O})^{*}\tilde{t}_k \zeta G(\mathscr{O})^{*}\right).
\end{align*}
In the following steps, we prove that for a fixed coweight $k$, the union over $\mu_n(F)$ is disjoint.

\paragraph{Step 1:} For a fixed $k$, consider the subgroup
\begin{align}
    S_k := G(\mathscr{O}) \cap t_k G(\mathscr{O})t_k^{-1}.
\end{align}
Every element $g\in S_{k}$ can be written as $g = t_k g_0 t_k^{-1}$ for a unique $g_0\in G(\mathscr{O})$. We define a map
\begin{align}\label{define a new function}
    \Psi_{k} : S_{k} &\longrightarrow \mu_n(F) \\
    g &\longmapsto \zeta, \nonumber
\end{align}
where $\zeta$ is the unique element in $\mu_n(F)$ satisfying the relation $g^*\tilde{t}_k = \tilde{t}_k g_0^* \zeta$.

To see that $\Psi_{k}$ is a group homomorphism, let $g, h \in S_k$ with $\Psi_{k}(g)=\zeta_1$ and $\Psi_{k}(h)=\zeta_2$. A direct calculation shows that
\begin{align*}
    (gh)^*\tilde{t}_k &= g^*h^*\tilde{t}_k = g^*\tilde{t}_k h_0^* \zeta_2 = \tilde{t}_k g_0^* h_0^* \zeta_1 \zeta_2 = \tilde{t}_k (g_0 h_0)^* \zeta_1 \zeta_2.
\end{align*}
This implies $\Psi_{k}(gh) = \zeta_1 \zeta_2$, confirming the homomorphism property.

\paragraph{Step 2:} The homomorphism $\Psi_{k}$ is trivial if and only if the double cosets $G(\mathscr{O})^{*}\tilde{t}_k\zeta G(\mathscr{O})^{*}$ and $G(\mathscr{O})^{*}\tilde{t}_k\zeta' G(\mathscr{O})^{*}$ are disjoint for all pairs of distinct central elements $\zeta \neq \zeta'$ in $\mu_n(F)$.

Suppose these double cosets are mutually disjoint for distinct central elements. For any $g = t_k g_0 t_k^{-1} \in S_k$, the definition in \eqref{define a new function} implies
\begin{align*}
    g^*\tilde{t}_k \in G(\mathscr{O})^{*}\tilde{t}_k G(\mathscr{O})^{*} \cap G(\mathscr{O})^{*}\tilde{t}_k \Psi_k(g) G(\mathscr{O})^{*}.
\end{align*}
The disjointness assumption forces $\Psi_k(g) = 1$, meaning $\Psi_k$ is trivial.

Conversely, assume $\Psi_{k}$ is trivial, and suppose for contradiction that there exist distinct elements $\zeta \neq \zeta'$ such that 
\[
G(\mathscr{O})^{*}\tilde{t}_k\zeta G(\mathscr{O})^{*} \cap G(\mathscr{O})^{*}\tilde{t}_k\zeta' G(\mathscr{O})^{*} \neq \varnothing.
\]
Then there exist elements $g_1, g_2, h_1, h_2 \in G(\mathscr{O})$ satisfying $g_1^*\tilde{t}_k \zeta g_2^* = h_1^*\tilde{t}_k \zeta' h_2^*$, which can be rearranged into
\begin{align*}
    g^*\tilde{t}_k = \tilde{t}_k h^* \zeta'\zeta^{-1},
\end{align*}
for some $g^*, h^* \in G(\mathscr{O})^*$. It follows that $\Psi_{k}(g) = \zeta'\zeta^{-1} \neq 1$, contradicting the triviality of $\Psi_{k}$.

\paragraph{Step 3:} We finally show that the homomorphism $\Psi_{k}$ is identically trivial.

The group $S_k$ is generated by its unipotent intersections, namely $U_+ \cap S_k$ and $U_- \cap S_k$. Therefore, it suffices to check that $\Psi_k$ restricts to the identity on both $U_+ \cap S_k$ and $U_- \cap S_k$. This triviality follows directly from the relations established in Section~\ref{torus normalize unipotent}.
\end{proof}

\section{Hecke Algebra}\label{hecke algebra}

Fix a faithful embedding $\varepsilon:\mu_n(F)\hookrightarrow \mathbb C^{\times}$. We define the genuine Hecke algebra with respect to the congruence subgroup $K_m^*$ for the covering group $\widetilde{G}(F)$ as follows:
\begin{align*}
\mathcal{H}_{\varepsilon} (\widetilde{G}(F),K_m^* )
:=\left\{
f :\widetilde{G}(F) \longrightarrow \mathbb C \,:\,
\begin{aligned}
&\bullet\ \text{$f$ is locally constant and compactly supported},\\
&\bullet\ f(\zeta g)=\varepsilon(\zeta)f(g),
\quad \forall \zeta \in \mu_n(F),\ g\in \widetilde{G}(F),\\
&\bullet\ f(k_1gk_2)=f(g),
\quad \forall k_1,k_2\in K_m^*,\ g \in \widetilde{G}(F).
\end{aligned}
\right\}.
\end{align*}

The multiplication is given by the convolution of functions; that is,
\begin{align*}
    (f * h)(x)=\int_{\gf} f(y) h(y^{-1}x) dy ,         
\end{align*}
where $dy$ is the Haar measure of the group $\gf$, which is normalized such that the measure of $\mu_n(F)\times K_m^*$ is $1$.

\qquad Define
    \begin{align*}
      \widetilde{G}_o := \bigsqcup_{\substack{k \geq0 }}  G(\mathscr O)^{*}\tilde{t}_k G(\mathscr O)^{*}.
    \end{align*}
    The union is disjoint by Proposition~\ref{cartan}. Let $X(m)$ be the discrete set of $K_m^*$ double cosets $K_m^*\backslash \widetilde{G}_o/ K_m^*$. For each $K_m^*g K_m^*\in X(m)$, define the functions 
\begin{equation}\label{basis 2}
    \psi_{K_m^*gK_m^*}(h)=\begin{cases}
        \varepsilon(\zeta), &\text{ if } h\in K_m^*g\zeta K_m^*, \text{ for some }\zeta\in\mu_n(F),\\
        0, & \text{ otherwise}.
    \end{cases}
\end{equation} 

\begin{proposition}\label{basis of hecke algebra}
    The functions $\{\psi_{K_m^*gK_m^*}\}_{K_m^*gK_m^*\in X(m)}$ form a basis for the  $\mathbb C$-vector space $\mathcal{H}_\varepsilon (\widetilde{G}(F),K_m^* )$.
\end{proposition}
\begin{proof}
    Let $f\in \mathcal{H}_{\varepsilon} (\widetilde{G}(F),K_m^* )$. Since $f$ is $K_m^*$-bi-invariant, its support is stable under left and right multiplication by $K_m^*$; that is,
    \begin{align*}
        \text{Supp}(f)=K_m^*\text{Supp}(f)K_m^*=\cup_{x\in \text{Supp}(f)} K_m^*x K_m^*.
    \end{align*}
    Since $\text{Supp}(f)$ is compact, there are $x_1,x_2,\dots,x_n\in \text{Supp}(f)$ such that $\text{Supp}(f)=\cup_{i=1}^{n}K_m^* x_i K_m^*$. We can choose these double cosets $K_m^*x_iK_m^*$ to be distinct. Also, for any $\zeta\in\mu_n(F)$, we have $f(\zeta g)=\varepsilon(\zeta)f(g)$. Hence, if $K_m^*gK_m^*$ belongs to $\text{Supp}(f)$, then $K_m^*gK_m^*\zeta$ also belongs to $\text{Supp}(f)$ for all $\zeta\in \mu_n(F)$. Therefore, $\text{Supp}(f)=\cup_{i=1}^{m}\cup_{\zeta\in\mu_n(F)}K_m^*y_iK_m^* \zeta$, where each $y_i$ belongs to $\widetilde{G}_o $. Thus, 
\begin{align*}
    f=\sum_{i=1}^{m}f(y_i)\psi_{K_m^*y_iK_m^*}.
\end{align*}
Now, since distinct double cosets in $X(m)$ are disjoint, the supports of the functions $\psi_{K_m^*gK_m^*}$ are disjoint. Therefore, the family $\{\psi_{K_m^*gK_m^*}\}_{K_m^*gK_m^*\in X(m)}$ is linearly independent.
\end{proof}

\begin{theorem}\label{finitely presented algebra}
     The Hecke algebra $\mathcal{H}_\varepsilon (\widetilde{G}(F),K_m^* )$  is a finitely presented $\mathbb {C} $-algebra. 
\end{theorem}

  It has been remarked in \cite[\S 2]{BDKV} that the proof of the Hecke algebra being finitely presented for connected reductive groups will also work for its central extensions. We fill in the details here. Before providing a sketch of the proof, we introduce some preliminary notations and definitions.

Let $\rg$ denote the category of smooth genuine representations of $\gf$, and let $\mathfrak{Z}$ be the center of the abelian category $\rg$.

\textit{Parabolic subgroups} $\widetilde{P}$ in $\gf$ are precisely the inverse images of parabolic subgroups $P$ in $G$. Let $P=L\ltimes U$ be a parabolic subgroup where $L$ is the Levi subgroup and $U$ is the unipotent radical. Then $U$ splits canonically in $\gf$, yielding 
\begin{align*}
    \widetilde{P}= \widetilde{L} \ltimes U. 
\end{align*}
Let $\delta_P$ be the modulus character of $P$. Then $\delta_{\widetilde{P}} := \delta_P \circ \pi$ is the modulus character of $\widetilde{P}$.

Following the formulation in \cite[Section~1.8]{Bernstein-Zelevinsky} (see also \cite[\S 3]{KaplanDani}), we can define the normalized induction functor $\iota_{\widetilde{P}}^{\gf}$ and the Jacquet functor $r_{\widetilde{P}}^{\gf}$ for the covering group $\gf$. Additionally, these functors satisfy the necessary hypotheses for the geometric lemma \cite[Theorem~5.2]{Bernstein-Zelevinsky}. For a detailed study of these functors on coverings of the symplectic groups, one can refer to the work of Hanzer and Mui\'c \cite{Parabolic}.

Following the standard definition in the linear case \cite[\S IV.2]{renard}, a genuine representation $W$ of $\gf$ is said to be \textit{supercuspidal} if its Jacquet module $W_U$ vanishes for every proper parabolic subgroup $\widetilde{P} = \widetilde{L}U$ of $\gf$.

Let $G(F)^\circ = \{g \in G(F) : \text{val}_F(\chi(g)) = 0, \enspace \forall \chi \in \text{Hom}_F(G, G_m)\}$. Let $\gf^{\circ}$ be the inverse image of $G(F)^\circ$. Then the group $X_{\mathrm{nr}}(\gf)$ of \textit{unramified characters} of $\gf$ is defined by
\[
X_{\mathrm{nr}}(\gf) = \mathrm{Hom}(\gf/\gf^{\circ},\mathbb{C}^{\times}).
\]

As noted in \cite[\S3]{KaplanDani}, Harish-Chandra's characterization theorem for supercuspidal representations \cite[3.20--3.24]{BernsteinZelevinsky1976} also remains valid for $\gf$.

Let $\mathrm{Irr}(\gf)$ denote the set of isomorphism classes of irreducible representations of $\gf$, and let $\mathrm{Irr}(\gf)_{\mathrm{sc}} \subset \mathrm{Irr}(\gf)$ denote the subset consisting of supercuspidal representations.

Let $\widetilde{L}$ be a Levi subgroup of $\gf$ (possibly $\widetilde{L}=\gf$), and let $\widetilde{\rho}\in \mathrm{Irr}(\gf)_{\mathrm{sc}}$. The pair $(\widetilde{L},\widetilde{\rho})$ is called a \textit{supercuspidal pair}.

Now let $(\widetilde{L},\widetilde{\rho})$ and $(\widetilde{M},\widetilde{\tau})$ be supercuspidal pairs. We say that $(\widetilde{L},\widetilde{\rho})$ and $(\widetilde{M},\widetilde{\tau})$ are \textit{conjugate} if
\[
\tm={}^{g}\tl  \quad \text{and} \quad \widetilde{\tau} \cong \chi{\widetilde{\rho}}^{g},
\]
for some $g\in\gf$ and $\chi\in X_{nr}(\gf)$.
This defines an equivalence relation on the set of all supercuspidal pairs. We denote the equivalence class of $(\widetilde{L},\widetilde{\rho})$ by $[\widetilde{L},\widetilde{\rho}]$, and write $\mathfrak{B}(\gf)$ for the set of all equivalence classes of supercuspidal pairs.

\textit{The Bernstein block} $\mathcal{R}(\gf){(\widetilde{L},\widetilde{\rho})}$ is the subcategory of $\mathcal{R}(\gf)$ formed by representations $W$ that embed into a direct sum of representations $\iota_{\widetilde{P}}^{\gf}(W_{\widetilde{P}})$ taken over all parabolic subgroups $\widetilde{P}$ with Levi component $\widetilde{L}$, where each irreducible subquotient of $W_{\widetilde{P}}$ is isomorphic to $\chi\widetilde{\rho}$ for some $\chi\in X_{nr}(\gf)$.

We denote the center of the category $\mathcal{R}(\gf){(\widetilde{L},\widetilde{\rho})}$ by $\mathfrak{Z}{(\widetilde{L},\widetilde{\rho})}$.

Let $(\widetilde{L},\widetilde{\rho})$ be a supercuspidal pair. For each element $w \in N(\widetilde{L})/\widetilde{L}$, choose a representative $n \in N(\widetilde{L})$. We define the conjugate representation $w\cdot\widetilde{\rho}$ of $\widetilde{L}$ by
\[
(w\cdot\widetilde{\rho})(l)=\widetilde{\rho}(n^{-1}ln), \qquad l\in \widetilde{L}.
\]
The relative Weyl group $W(\widetilde{L},\widetilde{\rho})$ associated with the pair is defined as:
\[
W(\widetilde{L},\widetilde{\rho}) = \{\, w\in N(\widetilde{L})/\widetilde{L} \mid w\cdot \widetilde{\rho}=\widetilde{\rho} \,\}.
\]

A representation $V$ of $\gf$ is called \textit{admissible} if, for every open compact subgroup $\widetilde{K}\subset \gf$, the space of $\widetilde{K}$-fixed vectors
$
V^{\widetilde{K}} = \{\,v\in V : kv=v \text{ for all } k\in \widetilde{K}\,\}
$ is finite-dimensional. A representation $V$ of $\gf$ is said to be of finite type if there exists a finite collection of vectors $\{v_i\}_{i\in I}\subset V$ such that the set $\{\, g v_i : g\in \gf,\ i\in I \,\}$ spans $V$ as a complex vector space.

Let $B$ be a $\mathbb{C}$-algebra, and let $V$ be a $B$-module equipped with an action of $\gf$. The representation $V$ is called $B$-admissible if, for every open compact subgroup $\widetilde{K}\subset \gf$, the space $V^{\widetilde{K}}$ of $\widetilde{K}$-fixed vectors is a $B$-module of finite type. The representation $V$ is said to be of $B$-finite type if there exists a finite family of vectors $\{v_i\}_{i\in I}\subset V$ such that $\{\, g v_i : g\in \gf,\ i\in I \,\}$ generates $V$ as a $B$-module.

\subsection{Proof of Theorem~\ref{finitely presented algebra}} The proof relies on the following three results:

\begin{itemize}[leftmargin=4cm]
    \item[\text{ (i)}] (Bernstein decomposition) $\mathcal{R}(\gf)=\underset{(\tl,\widetilde{\rho})\in\mathfrak{B}(\gf)}{\bigoplus}\mathcal{R}(\gf){(\tl,\widetilde{\rho})}.$

    \item[\text{ (ii)}] $\mathfrak{Z}{(\tl,\widetilde{\rho})}\cong \text{ Ring of regular functions on  } \widetilde{\rho}/W(\tl,\widetilde{\rho})$.

     \item[\text{(iii)}] Every representation of $\gf$ of finite type is $\mathfrak{Z}$- admissible. 
\end{itemize}

We briefly explain why these three results imply that the Hecke algebra is finitely presented.

By (i), we obtain the decomposition
\[
\mathfrak{Z}
=
\bigoplus_{(\tl,\widetilde{\rho})\in \mathfrak{B}(\gf)}
\mathfrak{Z}{(\tl,\widetilde{\rho})}.
\]

Let $L_c^\infty(\gf/K_m^*)$ denote the space of compactly supported genuine functions on $\gf$ that are right $K_m^*$-invariant. The group $\gf$ acts on $L_c^\infty(\gf/K_m^*)$ via the regular representation. Moreover, this representation is generated by a single element, namely the function $\psi_{K_m^*}$.

By (iii), the representation $L_c^\infty(\gf/K_m^*)$ is $\mathfrak{Z}$-admissible. Consequently,
\[
L_c^\infty(\gf/K_m^*)^{K_m^*}
=
\mathcal{H}_\varepsilon (\widetilde{G}(F),K_m^* )
\]
is a finitely generated $\mathfrak{Z}$-module. Hence, only finitely many Bernstein blocks $\mathcal{R}(\gf){(\tl,\widetilde{\rho})}$ act nontrivially on $\mathcal{H}_\varepsilon (\widetilde{G}(F),K_m^* )$. Therefore, $\mathfrak{Z}$ acts on $\mathcal{H}_\varepsilon (\widetilde{G}(F),K_m^* )$ via the finite quotient $\prod_{i=1}^n \mathfrak{Z}{(\tl_i,\widetilde{\rho}_i)}.$

By (ii), each algebra $\mathfrak{Z}{(\tl,\widetilde{\rho})}$ is the coordinate ring of an affine variety. In particular, every $\mathfrak{Z}{(\tl,\widetilde{\rho})}$ is a finitely generated $\bc$-algebra. It follows that $\prod_{i=1}^n \mathfrak{Z}{(\tl_i,\widetilde{\rho}_i)}$ is also a finitely generated $\bc$-algebra. Hence, by Hilbert, this algebra is finitely presented over $\bc$.

Since $\mathcal{H}_\varepsilon (\widetilde{G}(F),K_m^* )$ is a finitely generated module over the finitely presented algebra $\prod_{i=1}^n \mathfrak{Z}{(\tl_i,\widetilde{\rho}_i)}$, we conclude that $\mathcal{H}_\varepsilon (\widetilde{G}(F),K_m^* )$ is itself a finitely presented $\bc$-algebra.

\qquad The proof of (i) is established in \cite[Theorem 3.9]{KaplanDani}. For (ii), the proof for reductive groups given in \cite[\S 2]{BDKV} relies on the geometric lemma for induction and Jacquet functors, alongside technical results from \cite[\S 1]{BDKV}. Since Section 1 of \cite{BDKV} holds for arbitrary locally compact topological groups, and the geometric lemma extends to central extensions of reductive groups, the same argument remains valid for central extensions. The proof of (iii) is identical to the reductive group case proved in \cite[\S 3]{BDKV}.

\section{Identification of the Quotients of $\widetilde{G}(\mathscr{O})$ and $\widetilde{G}(\mathscr{O}')$}\label{close local fields}

Let $F'$ be another local non-archimedean field with the ring of integers $\mathscr{O'}$ and the maximal ideal $\mathscr{P'} \subset \mathscr{O'}$. Assume $F'$ is  $m$-close to $F$. We denote by \(\delta'\) the Kubota cocycle associated with \(F'\), and by $S'$ the splitting of \(\widetilde{G}(\mathscr{O}')\) such that $S'(g')=(g',\theta'(g'))$. 

\qquad Following this, we obtain an isomorphism between the residue fields, which further restricts to an isomorphism $\phi$ from $\mu_{q-1}(F)$ to $\mu_{q-1}(F')$. Since $F$ and $F'$ are $m$-close, we have $\mathscr{O}/\mathscr{P}^m\cong\mathscr{O'}/\mathscr{P'}^m$. For the corresponding congruence subgroups $K_m\subset G(\mathscr{O})$ and $K_m'\subset G(\mathscr{O'})$, we have an isomorphism $\Phi_m:G(\mathscr{O})/K_m\rightarrow G(\mathscr{O'})/K_m'$ which maps $(t+\mathscr{P}^m)$ to $t'+\mathscr{P'}^m$, where $t'$ is a uniformizer in $F'$. 

\subsection{}\label{valuation correspondence}Each coset $K_m g$ admits a representative $g = (g_{ij}) \in G(\mathscr{O})$ such that every entry $g_{ij}$ is either $0$ or of the form
\[
g_{ij} = u_{ij} t^k,
\]
with $u_{ij}=\omega(u_{ij})v_{ij} \in \mu_{q-1}(F)\times(1+\mathscr{P})$ and $0\leq k < m$.

Let $g' = (g'_{ij}) \in G(\mathscr{O}')$ that satisfies
\[
\Phi_m:(K_m^* g) = K_m'^* g',
\]
where the entries of $g'$ are defined as follows:
\[
g'_{ij} =
\begin{cases}
0, & \text{if } g_{ij} = 0, \\
u'_{ij} {t'}^k, & \text{if } g_{ij} = u_{ij} t^k,
\end{cases}
\]
and $u'_{ij}= \phi\big(\omega(u_{ij})\big)v_{ij}' \in \mu_{q-1}(F')\times (1+\mathscr{P'})$, where $v_{ij}$ maps to $v'_{ij}$ under the isomorphism $1+\mathscr{P}/1+\mathscr{P}^m\rightarrow 1+\mathscr{P'}/1+\mathscr{P'}^m$.

\qquad That is, $v(g_{ij})=v'(g'_{ij})$ if $g_{ij}\neq 0$ and $v(g_{ij})< m$.

\subsection{}
Let $g=(g_{ij})\in G(F) $ and $g'=(g'_{ij})\in G(F')$. We write $g\sim g'$ if
\[
g'_{ij} =
\begin{cases}
0, & \text{if } g_{ij} = 0, \\
u'_{ij} {t'}^k, & \text{if } g_{ij} = u_{ij} t^k,
\end{cases}
\]
where $u_{ij}\in \mathscr{O^\times}$ and $u'_{ij}\in \mathscr{O'}^\times$ with $\omega'(u'_{ij}) = \phi\big(\omega(u_{ij})\big)$.

\begin{proposition}\label{corro}
    Let $g,h\in G(F)$ and $g',h'\in G(F')$. If $g\sim g'$, $h\sim h'$ and $gh\sim g'h'$, then 
    \begin{align*}
        \delta'(g',h')=\phi(\delta(g,h)).
    \end{align*}
\end{proposition}
\qquad The definition of the Kubota cocycle and the splitting is expressed in terms of the Hilbert symbol. Therefore, the cocycle $\delta$ depends only on the valuations of the entries and on the Teichmüller lifts of the units. In particular, we have
\begin{align*}
   (u'_{ij}{t'}^{r},u'_{kl}{t'}^s)_{F'}
   &=\omega'[(-1)^{rs} (u'_{kl})^{r}(u'_{ij})^{-s}]^{\frac{q-1}{n}}\\
   &=\phi\!\left(\omega[(-1)^{rs} (u_{kl})^{r}(u_{ij})^{-s}]^{\frac{q-1}{n}}\right)\\
   &=\phi\big((u_{ij}t^{r},u_{kl}t^{s})_{F}\big).
\end{align*}

\begin{remark}\label{corro 2}
    Similarly, as above, if $g\sim g'$ with $g \in G(\mathscr{O})$ and $g'\in G(\mathscr{O'})$, then $\theta(g)=\theta'(g')$.
\end{remark}

\begin{proposition}\label{bij}
    For $m\geq 1$, the following map 
    \begin{align*}
      p: G(\mathscr{O})^*/ K_m^* & \longrightarrow G(\mathscr{O}) / K_m\\
        K_m^*g^* & \longmapsto K_mg
    \end{align*}
        
    is a group isomorphism.
\end{proposition}
\begin{proof}
Suppose ${K_m^*}g^*=K_m^*h^*$. Then, $(g,\theta(g))(h^{-1},\theta(h^{-1}))=(gh^{-1},\theta(gh^{-1}))$ belongs to $K_m^*$. Therefore, $gh^{-1}\in K_m$, which implies $K_mg=K_mh$. Hence, the map $p$ is well defined.

\qquad $p(K_m^* g^*\cdot K_m^*h^*)=p(K_m^* (gh)^*)=K_m(gh)=K_mg\cdot K_mh$, which shows that $p$ is a group homomorphism. Consider the map
\[
q : G(\mathscr{O}) / K_m \longrightarrow G(\mathscr{O})^{*} / K_m^{*},
\]
defined by
\[
q(K_m g) = K_m^{*} g^{*}.
\]
If $K_mg=K_mh$, then $gh^{-1}\in K_m $. From Section~\ref{kubota}, we have $\theta(gh^{-1})=1$. Therefore, $(gh^{-1},\theta(gh^{-1}))\in K_m^*$, which implies $K_m^*g^*=K_m^*h^*$. Also, it is clear that the map $q$ is also a group homomorphism.

The proposition now follows from the fact that \(p \circ q\) and \(q \circ p\) are identity maps on their respective domains.
\end{proof}

\begin{corollary}\label{closeness1}
    The map $\Phi_m: G(\mathscr{O})/K_m \rightarrow G(\mathscr{O}')/K_m'$, together with the maps $p,q$ from Proposition~\ref{bij}, induces an isomorphism, also denoted as $\Phi_m$:
\[
\Phi_m : G(\mathscr{O})^{*}/K_m^{*} \to G(\mathscr{O}')^{*}/K_m'^{*}.
\]

\end{corollary}
\begin{proof}
 We consider the following sequence of maps
\begin{align*}
    G(\mathscr{O})^*/ K_m^* \overset{p} \longrightarrow G(\mathscr{O}) / K_m \overset{\Phi_m}\longrightarrow G(\mathscr{O'})/ K_m' \overset{q'} \longrightarrow G(\mathscr{O'})^* / K_m'^*,
\end{align*} where $p(K_m^*g^*)=K_mg$ and $q'(K_m'g')=K_m'^*{g'}^*$ are given above. Each of these maps is a group isomorphism, which gives the isomorphism $\Phi_m$ from  $G(\mathscr{O})^*/ K_m^*$ to $G(\mathscr{O'})^* / K_m'^*$.
\end{proof}

The congruence subgroup $K_m^*$ is normal in $\go$ for $m\geq 1$. To see this, we start with the elements $(g,a)\in \go$ and $(k,1)\in K_m^*$. Since $K_m\trianglelefteq G(\mathscr{O})$, $gkg^{-1}\in K_m$. 
\begin{align*}
    (g,a)(k,1)(g,a)^{-1}=(g,a)(k,1)(g^{-1},a^{-1}\delta(g,g^{-1})^{-1})
    = (gkg^{-1},\delta(g,k)\delta(gk,g^{-1})\delta(g,g^{-1})^{-1}).
\end{align*}

Recall from Section~\ref{kubota} that the map $\theta$ satisfies the relation $\theta(gh) = \theta(g)\theta(h)\delta(g,h)$ for all $g,h \in G(\mathscr{O})$, and restricts to $\theta(k)=1$ for all $k \in K_m$. Expanding the cocycle terms using these properties yields
\begin{align*}
\delta(g,k)\delta(gk,g^{-1})\delta(g,g^{-1})^{-1} &= \big(\theta(gk)\theta(g)^{-1}\theta(k)^{-1}\big) \big(\theta(gkg^{-1})\theta(gk)^{-1}\theta(g^{-1})^{-1}\big) \big(\theta(gg^{-1})\theta(g)^{-1}\theta(g^{-1})^{-1}\big)^{-1} \\
&= \theta(gkg^{-1}) \theta(k)^{-1} \theta(gg^{-1})^{-1} \\
&= \theta(gkg^{-1}) = 1,
\end{align*}
where the final equality holds because $gkg^{-1} \in K_m$. Hence, we have
\[
(g,a)(k,1)(g,a)^{-1} = (gkg^{-1}, 1) \in K_m^*,
\]
which establishes that $K_m^*$ is a normal subgroup of $\widetilde{G}(\mathscr{O})$ for all $m\geq 1$.

\begin{corollary} \label{closeness2}
    The isomorphism $\Phi_m$ in Corollary~\ref{closeness1} induces an isomorphism from $\widetilde{G}(\mathscr{O})/K_m^*$ to $\widetilde{G}(\mathscr{O'})/K_m'^*$.
    \end{corollary} 
   \begin{proof}
     Recall that $\widetilde{G}(\mathscr{O}) = \{ (g,a) : g \in G(\mathscr{O}), \, a \in \mu_n(F) \}$. Because $G(\mathscr{O})^*$ is the image of $G(\mathscr{O})$ under the homomorphism $S$, we obtain the coset decomposition $\widetilde{G}(\mathscr{O}) = \bigsqcup_{\zeta \in \mu_n(F)} G(\mathscr{O})^* \zeta$. Using this decomposition, we define the following map:
\begin{align*}
\Phi_m : \widetilde{G}(\mathscr{O})/K_m^* &\longrightarrow \widetilde{G}(\mathscr{O}')/K_m'^* \\
K_m^* (g,\theta(g)) \zeta &\longmapsto K_m'^* (g',\theta'(g')) \phi(\zeta),
\end{align*}
where the element $K_m^*(g,\theta(g))$ maps to $K_m'^*({g'},\theta'(g'))$ under the induced map $\Phi_m : G(\mathscr{O})^*/K_m^* \rightarrow G(\mathscr{O}')^*/K_m'^*$ from Corollary~\ref{closeness1}, and $\phi : \mu_{q-1}(F) \to \mu_{q-1}(F')$ is the isomorphism. The corollary then follows immediately from Corollary~\ref{closeness1} together with the fact that $\phi$ is an isomorphism.
   \end{proof}

\begin{remark}\label{not well defined Kubota}
For $m\geq 1$, the Kubota cocycle does not satisfy the property
\[
\delta(K_m g K_m,\, K_m h K_m) = \delta(g,h)
\]
for all $g,h \in \widetilde{G}(\mathscr{O})$.
\end{remark}
\begin{proof}
    It suffices to prove that there exist $g\in \widetilde{G}(\mathscr{O})$ and $k\in K_m$ such that
\[
\delta(g^{-1},gk)\neq\delta(g^{-1},g).
\]

Suppose for some $g \in \widetilde{G}(\mathscr{O}) $ and $k\in K_m$, we have that $\delta(g^{-1},gk)=\delta(g^{-1},g)$. Then, from the $2$ -cocycle relation,
\[
\delta(g^{-1},g)\delta(1,k)=\delta(g,k)\delta(g^{-1},gk),
\]
we obtain that 
\begin{align*}
    \delta(g,k)=\delta(1,k)=1.
\end{align*}

Therefore, it is enough to show that there exist elements $g\in G(\mathscr{O})$ and $k\in K_m$ such that $\delta(g,k)\neq 1$.
Now consider the elements
\begin{align*}
    g=
\begin{pmatrix}
u & 0\\
0 & u^{-1}
\end{pmatrix}
\, \quad \text{ and }\quad k=
\begin{pmatrix}
1+t^m x & t^m y\\
t^m w & 1+t^m z
\end{pmatrix}, 
\end{align*}
with $u\in \mathscr{O}^{\times}$, $x,y,z,w\in \mathscr{O}$, and $w\neq 0.$

For these elements, we have
\[
X(g)=u^{-1}\quad \text{and}\quad X(k)=t^m w.
\]

Multiplying $g$ and $k$ gives
\[
gk=
\begin{pmatrix}
* & *\\
u^{-1}t^m w & *
\end{pmatrix}.
\]
Therefore, $X(gk)=u^{-1}t^m w.$ Hence,
\[
\delta(g,k)
=
\left(
\frac{u^{-1}t^m w}{u^{-1}},
\frac{u^{-1}t^m w}{t^m w}
\right)_F
=
(t^m w, u^{-1})_F .
\]

Write $w=t^r v$ with $v\in \mathscr{O}^{\times}$. Then
\begin{align*}
\delta(g,k)
&=
(t^{m+r}v,u^{-1})_F 
=
(t,u^{-1})_F^{\,m+r} =
\omega(u^{-1})^{\frac{(q-1)(m+r)}{n}} .
\end{align*}

Finally, we can choose $r,u$ in such a way that $\omega(u^{-1})^{\frac{(q-1)(m+r)}{n}}\neq 1$, which yields the desired result.
\end{proof}

\section{Linear Isomorphism of Hecke Algebras over two Close Local Fields }

 For any $k \in \bz_{\ge 0}$ and for the local field $F$ with the ring of integers $\mathscr{O}$, we introduce the following notation: 
\begin{enumerate}

    \item $\widetilde{G}_{k}:=\widetilde{G}(\mathscr O)\tilde{t}_k \widetilde{G}(\mathscr O)$.
    \item $X_{k}(m):=$ discrete set of $K_m^*$ double cosets of $\widetilde{G}_k$.
\end{enumerate}
\qquad 
The group $\widetilde{G}(\mathscr{O}) \times \widetilde{G}(\mathscr{O})$ acts transitively on $\widetilde{G}_{k}$ via
\[
(g_1,g_2)\cdot h \mapsto g_1 h g_2^{-1}.
\]
Consequently, the group $\widetilde{G}(\mathscr{O})/K_m^* \times \widetilde{G}(\mathscr{O})/K_m^*$ acts transitively on $X_{k}(m)$ by
\[
(K_m^*g_1,K_m^*g_2)\cdot K_m^*hK_m^* \mapsto K_m^*g_1 h g_2^{-1}K_m^* .
\]
We denote by $\Gamma_k(m)$ the stabilizer of $K_m^*\tilde{t}_kK_m^*$ under this action.

Similarly, for the local field $F'$ which is $m$-close to $F$, we will use the notations $K_m',\widetilde{G}_{k}',X_{k}'(m), \Gamma_k'(m)$.
\begin{lemma}
The stabilizer of $K_m^*\tilde{t}_kK_m^*$ is 
\begin{align*}
    \Gamma_k(m)=\Big\{(K_m^*g,K_m^*\tilde{t}_k^{-1} g \tilde{t}_k): g\in \widetilde{G}(\mathscr{O})\cap\tilde{t}_k\widetilde{G}(\mathscr{O})\tilde{t}_k^{-1}\Big/K_m^*\cap\tilde{t}_kK_m^*\tilde{t}_k^{-1}\Big\}.
\end{align*}
\end{lemma}
\begin{proof}
   If $h \in \widetilde{G}(\mathscr{O}) \cap \tilde{t}_k \widetilde{G}(\mathscr{O}) \tilde{t}_k^{-1}$, then 
$\tilde{t}_k^{-1} h \tilde{t}_k \in \widetilde{G}(\mathscr{O})$. Hence, it is immediate that
\[
(K_m^* h,\, K_m^* \tilde{t}_k^{-1} h \tilde{t}_k) \in \Gamma_k(m).
\]

Conversely, suppose $(K_m^* g, K_m^* h) \in \Gamma_k(m)$. Then

\[
K_m^* g \tilde{t}_k h^{-1} K_m^* = K_m^* \tilde{t}_k K_m^*. 
\]
Therefore, there exist $k_1, k_2 \in K_m^*$ such that
\[
g \tilde{t}_k h^{-1} = k_1 \tilde{t}_k k_2 .
\]
Equivalently,
\[
x := k_1^{-1} g = \tilde{t}_k k_2 h \tilde{t}_k^{-1} \in 
\widetilde{G}(\mathscr{O}) \cap \tilde{t}_k \widetilde{G}(\mathscr{O}) \tilde{t}_k^{-1}.
\]
It follows that
\[
K_m^* g = K_m^* x
\quad \text{and} \quad
K_m^* h = K_m^* \tilde{t}_k^{-1} x \tilde{t}_k ,
\]
where $x \in \widetilde{G}(\mathscr{O}) \cap \tilde{t}_k \widetilde{G}(\mathscr{O}) \tilde{t}_k^{-1}$.

\end{proof}

\begin{lemma}\label{stabilizer}
 Consider the isomorphism $\Phi_m:\widetilde{G}(\mathscr{O})/ K_m^*$ to $\widetilde{G}(\mathscr{O'})/ K_m'^*$ from Corollary~\ref{closeness2}. Suppose the map $\Phi_m$ sends $K_m^*g$ to $K_m'^*g'$. Then the following map:
    \begin{align*}
        \chi:\Gamma_k(m)&\longrightarrow {\Gamma}_{k}' (m)  \\
        (K_m^*g,K_m^*\tilde{t}_k^{-1} g \tilde{t}_k)&\longmapsto(K_m'^*g',K_m'^*(\tilde{t}_k^{-1} g \tilde{t}_k)')
        \end{align*}
    is a group isomorphism.
\end{lemma}
\begin{proof}
    We have 
\[
G(\mathscr{O})=\langle U_+(\mathscr{O}),U_-(\mathscr{O})\rangle .
\]
Hence, for any $g\in \widetilde{G}(\mathscr{O})$, we may write,

\begin{align}
    g=\prod_{i=1}^n\tilde{u}_\alpha(x_i)\zeta ,\quad\text{ with }\alpha\in \{\pm 1\},
\end{align}
for some $\zeta\in \mu_n(F)$ and $x_i\in \mathscr{O}$ and $\tilde{u}_\alpha(x_i)$ is defined in Section \ref{unipotent splitting}. From Section~\ref{torus normalize unipotent}, we also have
\begin{align}
    \tilde{t}_k^{-1}g\tilde{t}_k=\prod_{i=1}^n\tilde{u}_\alpha(t^{2\alpha k}x_i)\zeta.
\end{align}
From Section~\ref{valuation correspondence} and Corollary~\ref{closeness2}, it follows that
\begin{align*}
    g'=k_1'\prod_{i=1}^n\tilde{u}_\alpha(x_i')\phi(\zeta)\quad\text{ and } \quad (\tilde{t}_k^{-1}g\tilde{t}_k)'
=k_2'\prod_{i=1}^n\tilde{u}_\alpha(t'^{2\alpha k}x_i')\phi(\zeta),
\end{align*}

 for some $k_1',k_2'\in K_m'^*$, where $v(x_i)=v(x_i')$. Now,
\begin{align*}
(\tilde{t}_k^{-1} g \tilde{t}_k)'
&=k_2'\prod_{i=1}^n\tilde{u}_\alpha(t'^{2\alpha k}x_i')\phi(\zeta) =k_2'\tilde{t'_k}^{-1}\prod_{i=1}^n\tilde{u}_\alpha(x_i')\tilde{t'_k}\phi(\zeta) =k_2'\tilde{t'_k}^{-1}k_1'^{-1}g'\tilde{t'_k} .
\end{align*}
Hence, $(K_m'^*g',K_m'^*(\tilde{t}_k^{-1} g \tilde{t}_k)')
=(K_m'^*g',K_m'^*\tilde{t'_k}^{-1}{k_1'}^{-1}g'\tilde{t'_k})$ which lies in ${\Gamma}_k'(m)$. This shows that the map $\chi$ is well defined. Since $\chi$ is the restriction of the map $\Phi_m\times\Phi_m$, it follows immediately that $\chi$ is an injective group homomorphism.

To prove surjectivity, consider an element
\[
(K_m'^*h',K_m'^*\tilde{t'_k}^{-1}h'\tilde{t'_k})\in {\Gamma}_k'(m).
\]
Write $h'=\prod_{i=1}^n \tilde{u}_\alpha(y_i')\eta$. Then there exists an element $h=\prod_{i=1}^n \tilde{u}_\alpha(y_i)\zeta\in \widetilde{G}(\mathscr{O})$ with $v(y_i)=v(y_i')$ and $\phi(\zeta)=\eta$ such that
\[
\Phi_m(K_m^*h)=K_m'^*h'.
\]
Moreover,
\[
\tilde{t}_k^{-1} h \tilde{t}_k
=\prod_{i=1}^n\tilde{u}_\alpha(t^{2\alpha k}y_i)\zeta
\quad\text{and}\quad
\tilde{t'_k}^{-1} h'\tilde{t'_k}
=\prod_{i=1}^n\tilde{u}_\alpha(t'^{2\alpha k}y_i')\eta .
\]
Hence,
\[
\Phi_m(K_m^*\tilde{t}_k^{-1} h \tilde{t}_k)
=K_m'^*\tilde{t'_k}^{-1} h'\tilde{t'_k},
\]
which completes the proof.
\end{proof}

\subsection{}\label{double bijection}
Using Lemma~\ref{stabilizer} together with the orbit--stabilizer theorem, for each $k \ge 0$ we obtain the following bijection:
\begin{align*}
X_k(m) 
&\longrightarrow 
\frac{\widetilde{G}(\mathscr{O})/K_m^* \times \widetilde{G}(\mathscr{O})/K_m^*}{\Gamma_k(m)}
\longrightarrow
\frac{\widetilde{G}(\mathscr{O}')/K_m'^* \times \widetilde{G}(\mathscr{O}')/K_m'^*}{\Gamma_k'(m)}
\longrightarrow 
X_k'(m),\\
K_m^* g_1 \tilde{t}_k g_2 K_m^*
&\longmapsto
(K_m^* g_1,\, K_m^* g_2)\Gamma_k(m)
\longmapsto
(K_m'^* g_1',\, K_m'^* g_2')\Gamma_k'(m)
\longmapsto
K_m'^* g_1' \tilde{t}_k' g_2' K_m'^* .
\end{align*}

Moreover, via this correspondence, the element 
$K_m^* g_1^* \tilde{t}_k g_2^* \zeta K_m^*$ is mapped to 
$K_m'^* g_1'^* \tilde{t}_k' g_2'^* \phi(\zeta) K_m'^*$.
This immediately yields the following corollary.

\begin{corollary}\label{linear space isomorphism}
For $m$-close local fields $F$ and $F'$, we have a natural $\mathbb{C}$-vector space isomorphism between the genuine Hecke algebras of the corresponding metaplectic groups:
\[
\Phi_m:\mathcal{H}_\varepsilon(\widetilde{G}(F), K_m^*)
\longrightarrow
\mathcal{H}_{\varepsilon'}(\widetilde{G}(F'), K_m'^*).
\]
\end{corollary}

\section{Kazhdan Isomorphism}

\begin{theorem}\label{isomorphism}
Let $F$ be a non-archimedean local field and $l$ a positive integer. 
There exists an integer $m \ge l$ such that, for any non-archimedean 
local field $F'$ that is $m$-close to $F$, the vector space isomorphism 
\begin{equation*}
\Phi_l : \mathcal{H}_\varepsilon(\widetilde{G}(F), K_l^*) \longrightarrow \mathcal{H}_{\varepsilon'}(\widetilde{G}(F'), K_l'^*)
\end{equation*}
from Corollary~\ref{linear space isomorphism} is an isomorphism of algebras. 
Here, the metaplectic groups $\widetilde{G}(F)$ and $\widetilde{G}(F')$ are 
equipped with Kubota cocycles, and the character $\varepsilon'$ is given 
by $\varepsilon \circ \phi^{-1}$, where $\phi : \mu_{q-1}(F) \to \mu_{q-1}(F')$ 
is the isomorphism described in Section~\ref{close local fields}.
\end{theorem}

\qquad This theorem will follow from several technical lemmas presented below.
\begin{lemma}\label{natural}
    Let $C\subset \bz_{\geq 0} $ be a finite subset. Define
    \begin{align*}
        \widetilde{G}_{_C}=\underset{k\in C}{\bigsqcup}G(\mathscr{O})^*\tilde{t}_kG(\mathscr{O})^*.
    \end{align*}
    There exists a natural number $m=m_{_C}\geq l$ that depends only on $C$ but not on $F$ such that for all $g\in \widetilde{G}_{_C}$, we have $gK_m^*g^{-1}\subset K_l^*$ and $g^{-1}K_m^*g\subset K_l^*$.
\end{lemma}
\begin{proof}
   
For a positive integer $l$ and $\alpha \in \{\pm 1\}$, let $\widetilde{U}_{\alpha,l} = U_\alpha(\mathscr{P}^l)^*$ and $\widetilde{T}_l = \tilde{\lambda}_0(1+\mathscr{P}^l)$. The Iwahori factorization then yields
\begin{align*}
    K_l^* = \prod_{\alpha \in \{\pm 1\}} \widetilde{U}_{\alpha,l} \times \widetilde{T}_l.
\end{align*}

Then,
\begin{align*}
    \tilde{t}_kK_l^*\tilde{t}_k^{-1}=\underset{\alpha\in\{\pm 1\}}{\prod} \widetilde{U}_{\alpha,\,l-2\alpha k} \times  \widetilde{T_l}\quad
     \textit{and,}\quad\tilde{t}_k^{-1}K_l^*\tilde{t}_k=\underset{\alpha\in\{\pm 1\}}{\prod} \widetilde{U}_{\alpha,\,l+2\alpha k} \times  \widetilde{T_l}.
\end{align*}
For each $k\in C$, we will choose a large enough $m_k\geq l$ such that $m_k-2\alpha k,m_k+2\alpha k\geq l$. 

\qquad Then, $\tilde{t}_k{K_{m_k}^*}\tilde{t}_k^{-1},\, \tilde{t}_k^{-1}{K_{m_k}^*}\tilde{t}_k\subset K_l^* $. Finally, for $m=max\{m_{_k}:k\in C\}$, we have that $g{K_{_m}^*}g^{-1}\subset K_l^*$ and  $g^{-1}{K_{_m}^*}g\subset K_l^*$ for all $g\in \widetilde{G}_{_C}$.
\end{proof}

\begin{corollary}\label{biinv}
Suppose $F'$ is $m=m_C$-close to $F$, where $m_C$ comes from Lemma~\ref{natural}. 
For any $k,r\in C$ and $g_i^*=(g_i,\theta(g_i))\in G(\mathscr{O})^*$, the set
\[
g_1^*\tilde{t}_k g_2^*K_l^*g_3^*\tilde{t}_rg_4^*K_l^*
\]
is $K_m^*$-bi-invariant. The same holds over $F'$.
\end{corollary}

\begin{proof}
From Lemma~\ref{natural}, we have that $K_m^*(g_1^*\tilde{t}_k g_2^*)\subset g_1^*\tilde{t}_k g_2^* K_l^*$. That is,  
\[
g_1^*\tilde{t}_k g_2^*K_l^*g_3^*\tilde{t}_rg_4^*K_l^*
=
K_m^*g_1^*\tilde{t}_k g_2^*K_l^*g_3^*\tilde{t}_rg_4^*K_l^*.
\]
The same argument applies to $F'$.
\end{proof}

\begin{corollary}\label{uniondecomp}
With the above notation, for each $g_i^*\in G(\mathscr{O})^*$, we have
\[
g_1^*\tilde{t}_k g_2^*K_l^*g_3^*\tilde{t}_rg_4^*K_l^*
=
\bigcup_{k\in K_l^*/ K_m^*}
K_m^*g_1^*\tilde{t}_k g_2^*kg_3^*\tilde{t}_rg_4^*K_l^* .
\]
An analogous decomposition holds over $F'$.
\end{corollary}

\begin{proof}
Applying Corollary~\ref{biinv}, we have
\[
g_1^*\tilde{t}_k g_2^*K_l^*g_3^*\tilde{t}_rg_4^*K_l^*
=
K_m^*g_1^*\tilde{t}_k g_2^*K_l^*g_3^*\tilde{t}_rg_4^*K_l^* .
\]
Writing $K_l^*$ as a union of distinct $K_m^*$ cosets gives
\begin{align*}
    &K_m^*g_1^*\tilde{t}_k g_2^*
\left(\bigsqcup_{k\in K_l^*/K_m^*} k K_m^* \right)
g_3^*\tilde{t}_r g_4^*K_l^*\\
&=\bigcup_{k\in K_l^*/K_m^*}K_m^*g_1^*\tilde{t}_k g_2^* k K_m^*
g_3^*\tilde{t}_r g_4^*K_l^*.
\end{align*}
Since $r\in C$, applying the Lemma~\ref{natural} yields the desired expression:
\[
\bigcup_{k\in K_l^*/ K_m^*}
K_m^*g_1^*\tilde{t}_k g_2^*kg_3^*\tilde{t}_rg_4^*K_l^*.
\]
The same argument works for $F'$.

\end{proof}

\begin{corollary}\label{repindependence}
The set of distinct $K_m^*$-double cosets in $K_m^*g_1^*\tilde{t}_k g_2^*\tilde{t}_rg_3^*K_l^*$ is independent of the choice of representatives of the cosets
$K_m^*g_1^*$, $K_m^*g_2^*$, and $K_m^*g_3^*$.
\end{corollary}

\begin{proof}
Suppose
\[
g_i^*=h_i^*k_i \quad \text{for} \quad 1\leq i \leq 3,
\]
with $h_i^*\in G(\mathscr{O})^*$ and
$k_i\in K_m^*$.

Then
\begin{align*}
K_m^*g_1^*\tilde{t}_k g_2^*\tilde{t}_rg_3^*K_l^*
&=
K_m^*h_1^*k_1\tilde{t}_k h_2^*k_2\tilde{t}_rh_3^*k_3K_l^* \\
&=
K_m^*h_1^*\tilde{t}_k h_2^*k_2\tilde{t}_rh_3^*K_l^*  \\
&=
K_m^*h_1^*\tilde{t}_k h_2^*\tilde{t}_rh_3^*K_l^*,
\end{align*}
where the second equality uses the fact that $K_m^*\trianglelefteq G(\mathscr{O})^* $ and the last equality follows from Lemma~\ref{natural}.
Hence, the resulting set of double cosets does not depend on the chosen representatives.
\end{proof}

\begin{lemma}\label{doublecosetbijection}
Let $C \subset \mathbb{Z}_{\ge 0}$ be a finite subset. Let $k, r \in C$ and $g_i^* = (g_i, \theta(g_i)) \in G(\mathscr{O})^*$ for $1 \le i \le 4$. Assume that $F'$ is $m = m_C$-close to $F$. Let $g_i'^* = (g_i', \theta'(g_i'))$ be such that $\Phi_m(K_m g_i) = K_m' g_i'$ for each $i$. Then, the bijection from Section~\ref{double bijection} induces a bijection from the set of $K_m^*$ double cosets in 
\[
g_1^*\tilde{t}_k g_2^*K_l^*g_3^*\tilde{t}_r g_4^*K_l^*
\]
to the set of $K_m'^*$ double cosets in 
\[
g_1'^*\tilde{t'_k} g_2'^*K_l'^*g_3'^*\tilde{t'_r} g_4'^*K_l'^*.
\]
\end{lemma}

\begin{proof}
From Corollary~\ref{uniondecomp}, we have that it suffices to show the bijection between $K_m^*$-double cosets in $K_m^*g_1^*\tilde{t}_kg_2^*\tilde{t}_r g_3^*K_l^*$ and $K_m'^*$-double cosets in $K_m'^*g_1'^*\tilde{t'_k}g_2'^*\tilde{t'_r} g_3'^*K_l'^*$.

\qquad From Corollary~\ref{repindependence} and Section~\ref{valuation correspondence}, we can also assume that $g_i\sim g_i'$. By Cartan decomposition on a reductive group, we have 
\[
t_k g_2 t_r = h_1 t_s h_2 
\quad \text{and} \quad 
t'_k g'_2 t'_r=  h'_1 t'_s h'_2 ,
\]
for some $s\in \bz_{\geq0}$, where $h_i \in G(\mathscr{O})$ and $h'_i \in G(\mathscr{O}')$ satisfy
\[
\Phi_m(K_m h_i) = K'_m h'_i,
\]
and $h_i \sim h'_i$ for $i=1,2$.

Using Section~\ref{torus splitting}, we obtain
\[
\tilde{t}_k g_2^*\tilde{t}_r
=(t_k,\beta_{-k})(g_2,\theta(g_2))(t_r,\beta_{-r})=
\Big(t_k g_2 t_r,\, \beta_{-k}\beta_{-r}\theta(g_2)\delta(t_k,g_2)\delta(t_k g_2,t_r)\Big),
\]
and 
\[
h_1^*\tilde{t}_s h_2^*=(h_1,\theta(h_1))(t_s,\beta_{-s})(h_2,\theta(h_2))=\Big(h_1t_s h_2,\,\theta(h_1)\theta(h_2)\beta_{-s}\delta(h_1,t_s)\delta(h_1t_s,h_2)\Big).
\]
Therefore,
\begin{align*}
    \tilde{t}_k g_2^*\tilde{t}_r=h_1^*\tilde{t}_s h_2^*\zeta \quad \text{and hence,}\quad K_m^*g_1^*\tilde{t}_k g_2^*\tilde{t}_rg_3^*K_l^*=
K_m^*g_1^*h_1^*\tilde{t}_s h_2^*g_3^*K_l^*\zeta,
\end{align*}

where $\zeta=\beta_{-k}\beta_{-r}\beta_{-s}^{-1}\theta(g_2)\theta(h_1)^{-1}\theta(h_2)^{-1}\delta(t_k,g_2)\delta(t_k g_2,t_r)\delta(h_1,t_s)^{-1}\delta(h_1t_s,h_2)^{-1}$.

\qquad Similarly, for the local field $F'$, we also have  $K_m'^*g_1'^*\tilde{t'_k}g_2'^*\tilde{t'_r}g_3'^*K_l'^*
=
K_m'^*g_1'^*h_1'^*\tilde{t'_s}h_2'^*g_3'^*K_l'^*\eta$, where 
\begin{align*}
    \eta=\beta'_{-k}\beta'_{-r}\beta_{-s}^{'-1}\theta'(g_2')\theta'(h_1')^{-1}\theta'(h_2')^{-1}\delta'(t'_k,g_2')\delta'(t'_k g_2',t'_r)\delta'(h_1',t'_s)^{-1}\delta'(h'_1t'_s,h_2')^{-1}.
\end{align*}

For $j=-k,-r,-s$, the formula for $\beta'_j$ is the same as $\beta_j$ from Section~\ref{torus splitting}, where $t$ is replaced by $t'$. Hence, $\phi(\beta_i)=\beta'_i$.

We also have $g_2\sim g_2'$, $h_i\sim h_i'$, which implies $t_kg_2\sim t'_kg_2'$, $h_1t_s \sim h_1't_s'$ and $h_1t_s h_2=t_kg_2t_r\sim t'_kg'_2t'_r=h'_1t'_s h'_2$. Hence, using the compatibility of cocycles and splittings given in  
Proposition~\ref{corro} and Remark~\ref{corro 2}, we have  
\begin{align*}
    &\theta'(g_2')=\phi(\theta(g_2)),\, \theta'(h_i')=\phi(\theta(h_i)),\,\delta'(t'_k,g_2')=\phi(\delta(t_k,g_2)), \\
    &\delta'(t'_k g_2',t'_r)=\phi(\delta(t_k g_2,t_r)),\, \delta'(h_1',t'_s)^{-1}=\phi(\delta(h_1,t_s)^{-1}),\, \delta'(h'_1t'_s,h_2')^{-1}=\phi(\delta(h_1t_s,h_2)^{-1}).
\end{align*}
That is, $\eta=\phi(\zeta)$. Therefore, if
\begin{align*}
K_m^*g_1^*\tilde{t}_k g_2^*\tilde{t}_rg_3^*K_l^*=
K_m^*g_1^*h_1^*\tilde{t}_s h_2^*g_3^*K_l^*\zeta
\end{align*}
then, 
\begin{align*}
K_m'^*g_1'^*\tilde{t'_k}g_2'^*\tilde{t'_r}g_3'^*K_l'^*
&=
K_m'^*g_1'^*h_1'^*\tilde{t'_s} h_2'^*g_3'^*K_l'^*\phi(\zeta).
\end{align*}

Finally, if $K_l^*=\bigsqcup k_i^*K_m^*$, then
$K_l'^*=\bigsqcup k_i'^*K_m'^*$ with $\Phi_m(K_mk_i)=K_m'k_i'$;  which gives the desired bijection of
double cosets.
\end{proof}

 The Haar measure on $\gf$ is normalized so that $\mu_n(F)\times K_l^*$ has measure $1$. Similarly, the Haar measure on $\hf$ is normalized so that $\mu_n(F')\times K_l'^*$ has measure $1$.

\begin{lemma}\label{vol}
   If for any $g,h\in \gf$, we have 
    \begin{align*}
        Vol(K_l^*gK_l^*)Vol(K_l^*hK_l^*)=Vol(K_l^*ghK_l^*),
    \end{align*} then 
 \begin{align*}
    \psi_{_{K_l^*gK_l^*}} * \psi_{_{K_l^*hK_l^*}}=\psi_{_{K_l^*ghK_l^*}}.
 \end{align*}
    \end{lemma}
    \begin{proof}
      Write $K_m^*gK_m^*$ as a union of $K_m^*$-cosets
\begin{align*}
    K_m^*gK_m^*=\underset{k}{\bigcup}kgK_m^*.
\end{align*}
Suppose $\delta_x$ is the point mass function at $x$, then
\begin{align*}
    \psi_{_{K_m^*gK_m^*}}=\underset{k}{\sum}\delta_{kg} *\psi_{_{K_m^*}}.
\end{align*}
In a similar way, we have 
\begin{align*}
    \psi_{_{K_m^*hK_m^*}}=\underset{k'}{\sum}\delta_{k'h} *\psi_{_{K_m^*}}.
\end{align*}
Now, in the Hecke algebra $\mathcal{H}_\varepsilon(\widetilde{G}(F),K_m^*)$, the identity element is $\psi_{_{K_m^*}}$. Therefore, we have
\begin{align*}
    \psi_{_{K_m^*gK_m^*}}* \psi_{_{K_m^*hK_m^*}}&=(\underset{k}{\sum}\delta_{kg} *\psi_{_{K_m^*}})*(\underset{k'}{\sum}\delta_{k'h} *\psi_{_{K_m^*}})\\
    &=\underset{k,k'}{\sum}\delta_{kg} *\psi_{_{K_m^*}}*(\delta_{k'h} *\psi_{_{K_m^*}})\\
    &=\underset{k,k'}{\sum}\delta_{kgk'h} *\psi_{_{K_m^*}}.
\end{align*}
The rest of the proof then follows from \cite[Proposition 2.2]{howe}.  
    \end{proof}

\begin{lemma}\label{volume}
For the congruence subgroup $K_l^*$ of $G(\mathscr{O})^*$, we have the following:
\begin{enumerate}[label=\roman*.]
    \item For $k,r\in \bz_{\geq 0}$, we have $\psi_{_{K_l^*\tilde{t}_kK_l^*}} * \psi_{_{K_l^*\tilde{t}_rK_l^*}}=\psi_{_{K_l^*\tilde{t}_{k+r}K_l^*}}$.
    \item For $k\in \bz_{\geq 0}$ and $g_1,g_2\in G(\mathscr{O})^*$, we have $\psi_{_{K_l^*g_1\tilde{t}_k g_2K_l^*}}=\psi_{_{K_l^*g_1K_l^*}}*\psi_{_{K_l^*\tilde{t}_kK_l^*}}*\psi_{_{K_l^*g_2K_l^*}}$.
\end{enumerate}
\end{lemma}
\begin{proof}
To compute the volume $\text{Vol}(K_l^* \tilde{t}_k K_l^*)$, we first note that it is determined by the index of the intersection subgroup:
$$
\text{Vol}(K_l^* \tilde{t}_k K_l^*) = [K_l^* : \tilde{t}_k K_l^* \tilde{t}_k^{-1} \cap K_l^*].
$$

For a positive integer $l$ and $\alpha \in \{\pm 1\}$, let $\widetilde{U}_{\alpha, \,l} := \widetilde{U}_\alpha(\mathscr{P}^l)$ and $\widetilde{T}_l := \tilde{\lambda}_0(1 + \mathscr{P}^l)$. Applying the Iwahori factorization yields:
$$
K_l^* = \prod_{\alpha \in \{\pm 1\}} \widetilde{U}_{\alpha,\, l} \times \widetilde{T}_l.
$$

Then, we have
$$
\tilde{t}_k K_l^* \tilde{t}_k^{-1} = \prod_{\alpha \in \{\pm 1\}} \widetilde{U}_{\alpha, \,l - 2\alpha k} \times \widetilde{T}_l.
$$

This yields the quotient structure:
$$
K_l^* / \left( \tilde{t}_k K_l^* \tilde{t}_k^{-1} \cap K_l^* \right) \cong \prod_{\alpha \in \{\pm 1\}} \widetilde{U}_{\alpha, 
\,l} / \left( \widetilde{U}_{\alpha,\, l} \cap \widetilde{U}_{\alpha, \,l - 2\alpha k} \right).
$$

Taking the cardinality of this quotient yields the index:
$$
[K_l^* : \tilde{t}_k K_l^* \tilde{t}_k^{-1} \cap K_l^*] = \prod_{\alpha \in \{\pm 1\}} [\widetilde{U}_{\alpha,\, l} : \widetilde{U}_{\alpha,\, l} \cap \widetilde{U}_{\alpha,\, l - 2\alpha k}].
$$

For $\alpha = 1$, $\widetilde{U}_{\alpha,\,l}\subseteq \widetilde{U}_{\alpha,\,l-2\alpha k} $, and for  $\alpha =-1$, $\widetilde{U}_{\alpha,\,l-2\alpha k}\subseteq  \widetilde{U}_{\alpha,\,l}$. This gives us
\begin{align*}
    [K_l^* :  \tilde{t}_kK_l^*\tilde{t}_k^{-1}\cap K_l^*]= [\widetilde{U}_{-1,\,l}:\widetilde{U}_{-1,\,l+2\alpha k}].
\end{align*}

There is a bijection between the coset space $\widetilde{U}_{\alpha,\,l}/\widetilde{U}_{\alpha,\,l+2\alpha k}$ and $\mathscr{P}^l/\mathscr{P}^{l+2\alpha k}$. Hence,
\begin{align*}
    Vol(K_l^*\tilde{t}_kK_l^*)=[K_l^* :  \tilde{t}_kK_l^*\tilde{t}_k^{-1}\cap K_l^*]= q^{2\alpha k},
\end{align*}
where $q=[\mathscr{O}:\mathscr{P}]$. Therefore, for any $k,r\in \Lambda^+$, we have
\begin{align*}
    Vol(K_l^*\tilde{t}_kK_l^*)Vol(K_l^*\tilde{t}_rK_l^*)=Vol(K_l^*\tilde{t}_{k+r}K_l^*).
\end{align*}
The parts $i.$ and $ii.$ then follow from Lemma~\ref{vol}
\end{proof}

\begin{remark}\label{finitely presented}
    Choose a finite subset $\Lambda'\subset \bz_{\geq 0}$ such that $\Lambda'$ contains $0$, and generate $\Lambda^+$ as a semigroup. Fix a set of representatives $V_{\mathscr{O}^*}$ of $G(\mathscr{O})^*/K_l^*$. Then, by Lemma~\ref{volume}, the set $S=\{\psi_{_{K_l^*\tilde{t}_kK_l^*}}|k\in \Lambda'\}\cup \{\psi_{_{K_l^*gK_l^*}}|g\in V_{\mathscr{O}^*}\}$ generates the algebra $\mathcal{H}_\varepsilon(\tilde{G}(F),K_l^*)$.
\end{remark}

\begin{lemma}\label{algebra-iso}
   Suppose the local field $F'$ is $m$-close to $F$ for $m = m_C$, where $m_C$ comes from Lemma~\ref{natural}. Then, the map 
$$
\Phi_l : \mathcal{H}_\varepsilon(\widetilde{G}(F), K_l^*) \longrightarrow \mathcal{H}_{\varepsilon'}(\widetilde{G}(F'), {K_l'}^*)
$$
is a homomorphism when restricted to functions supported on $\widetilde{G}_C$. That is, for any $h_1, h_2 \in \mathcal{H}_\varepsilon(\widetilde{G}(F), K_l^*)$ supported on $\widetilde{G}_C$, we have
\begin{align*}
    \Phi_l(h_1 * h_2) = \Phi_l(h_1) * \Phi_l(h_2).
\end{align*}
\end{lemma}
\begin{proof}
    It is enough to prove the lemma for $\psi_{_{K_l^*gK_l^*}}*\psi_{_{K_l^*hK_l^*}}$ for each $g,h\in \widetilde{G}_{_C}$. 
\begin{align*}
    \psi_{_{K_l^*gK_l^*}}*\psi_{_{K_l^*hK_l^*}} (x)
    =\int_{\widetilde{G}(F)} \psi_{_{K_l^*gK_l^*}}(y) \psi_{_{K_l^*hK_l^*}}(y^{-1}x) dy.
\end{align*} 
From \eqref{basis 2}, we have
\begin{align*}
    \psi_{_{K_l^*gK_l^*}}(y)=\begin{cases}
        \varepsilon(\zeta), &y\in K_l^*g\zeta K_l^*, \\
        0, & \text{otherwise},
    \end{cases}\quad \text{ and }
    \hspace{.4cm}
    \psi_{_{K_l^*hK_l^*}}(y^{-1}x) =\begin{cases}
        \varepsilon(\zeta), & y\in xK_l^*h^{-1}\zeta^{-1}K_l^*,\\
        0, & \text{otherwise}.
    \end{cases}
\end{align*}

Therefore,
\begin{align*}
    \psi_{_{K_l^*gK_l^*}}*\psi_{_{K_l^*hK_l^*}} (x)
    &=\sum_{\zeta,\widetilde{\zeta}\in \mu_n(F)}\varepsilon({\zeta\widetilde{\zeta}})Vol(K_l^*g\zeta K_l^* \cap xK_l^*h^{-1}\widetilde{\zeta}^{-1}K_l^*)\\
    &=n\sum_{\zeta\in \mu_n(F)}\varepsilon(\zeta) Vol\Big({K_l}
    ^*gK_l^* \cap xK_l^*h^{-1}\zeta^{-1}K_l^*\Big).
\end{align*}

  This yields, 
\begin{align*}
    \psi_{_{K_l^*gK_l^*}}*\psi_{_{K_l^*hK_l^*}}
    =n\sum_{k\in\bz_{\geq 0}}\sum_{K_m^*g_1^*\tilde{t}_kg_2^*K_m^*}\Big[\sum_{\zeta\in \mu_n(F)}\varepsilon(\zeta) Vol\Big({K_l}
    ^*gK_l^* \cap g_1^*\tilde{t}_kg_2^*K_l^*h^{-1}K_l^*\zeta^{-1}\Big)\Big]\psi_{_{K_m^*g_1^*\tilde{t}_kg_2^*K_m^*}}.
\end{align*}
Also,
\begin{align*}
    &\psi_{_{{K_l'}^*g'{K_l'}^*}}*\psi_{_{{K_l'}^*h'{K_l'}^*}}\\ \vspace{.3cm}
    =n\sum_{k\in\bz_{\geq 0}}\sum_{K_m'^*g_1'^*\tilde{t'_k} g_2'^*K_m'^*}\Big[\sum_{\eta\in \mu_n(F')}(\varepsilon\circ\phi^{-1})(\eta)& Vol\Big({K_l'}
    ^*g'{K_l'}^* \cap g_1'^*\tilde{t'_k} g_2'^*{K_l'}^*h'^{-1}{K_l'}^*\eta^{-1}\Big)\Big]\psi_{_{K_m'^*g_1'^*\tilde{t'_k} g_2'^*K_m'^*}}\\
    =n\sum_{k\in\bz_{\geq 0}}\sum_{K_m'^*g_1'^*\tilde{t'_k}g_2'^*K_m'^*}\Big[\sum_{\zeta\in \mu_n(F)}\varepsilon(\zeta) Vol\Big({K_l'}
    ^*g'&{K_l'}^* \cap g_1'^*\tilde{t'_k} g_2'^*{K_l'}^*h'^{-1}{K_l'}^*\phi(\zeta^{-1})\Big)\Big]\psi_{_{K_m'^*g_1'^*\tilde{t'_k} g_2'^*K_m'^*}}.
\end{align*}

From Lemma~\ref{doublecosetbijection}, we have 
\begin{align*}
    Vol\Big({K_l}
    ^*gK_l^* \cap g_1^*\tilde{t}_kg_2^*K_l^*h^{-1}K_l^*\zeta^{-1}\Big)=Vol\Big({K_l'}
    ^*g'{K_l'}^* \cap g_1'^*\tilde{t'_k}g_2'^*{K_l'}^*h'^{-1}{K_l'}^*\phi(\zeta^{-1})\Big).
\end{align*}

Therefore,
\begin{align*}
    \Phi_m\big(\psi_{_{K_l^*gK_l^*}}*\psi_{_{K_l^*hK_l^*}}\big)=\psi_{_{{K_l'}^*g'{K_l'}^*}}*\psi_{_{{K_l'}^*h'{K_l'}^*}}.
\end{align*}  Since $\Phi_m$ agrees with $\Phi_l$ on $\mathcal{H}_\varepsilon(\widetilde{G}(F),K_l^*)$, we finally have $\Phi_l(h_1 *h_2)= \Phi_l(h_1)* \Phi_l(h_2)$.
\end{proof}

\subsection{Proof of Theorem \ref{isomorphism}} \label{proof of hecke lagebra isomorphism}The proof of this is similar to that of \cite[Theorem~4.1]{hecke}. 

By Theorem~\ref{finitely presented algebra}, the Hecke algebra $\mathcal{H}_\varepsilon(\widetilde{G}(F),K_l^*)$ is finitely presented. That is, we can write
\begin{align*}
    \mathcal{H}_\varepsilon(\widetilde{G}(F),K_l^*)\cong \frac{\mathbb C\langle x_1,x_2,\cdots,x_p\rangle}{\langle R_1,R_2,\cdots,R_q\rangle},
\end{align*}
where $\mathbb C\langle x_1,x_2,\cdots,x_p\rangle$ is the noncommutative polynomial ring in $p$ variables such that $R_i(x_1,x_2,\cdots,x_p)=0$ for $1\leq i \leq q$. \\
Suppose we index the elements of the set $S$ in Remark~\ref{finitely presented} as $f_1,f_2,\cdots,f_r$. These elements form a system of generators for $\mathcal{H}_\varepsilon(\widetilde{G}(F),K_l^*)$. Let $G_i, 1\leq i \leq p$ be polynomials in the $r$ variables such that 
\begin{align*}
    G_i(f_1,f_2,\cdots,f_r)=x_i,
\end{align*}
and $F_j, 1\leq j \leq r$ be polynomials in $p$ variables such that
\begin{align*}
    F_j(x_1,x_2,\cdots,x_p)=f_j.
\end{align*}
Let $N$ be the maximal degree of the polynomials $R_i(G_1,G_2,\cdots,G_p);1\leq i\leq q$ and $F_j(G_1,G_2,\cdots,G_p);1\leq j\leq r$. Let $D \subset \Lambda^+$ be a finite subset such that all possible products of $N$ terms of the $f_i$'s are contained in $\widetilde{G}_{_{D}}(F)$. Choose $n=n_{D}\geq l$ as in Lemma \ref{natural}. Suppose $F,F'$ are $n$-close; then we have an algebra homomorphism
\begin{align*}
   \tau: \mathbb C\langle x_1,x_2,\cdots,x_p\rangle &\longrightarrow \mathcal{H}_{\varepsilon'}(\widetilde{G}(F'),{K_l'}^*)\\
    x_i &\longmapsto G_i(f_1',f_2',\cdots,f_r'),
\end{align*}
where $f_i'=\Phi_l(f_i)$. We also have from Lemma \ref{algebra-iso} that
\begin{align*}
    \tau(R_j(x_1,x_2,\cdots,x_p))=0 , 1\leq j \leq q.
\end{align*}
This induces an algebra homomorphism 
\begin{align*}
    \mathscr{\bar{\tau}}: \mathcal{H}_\varepsilon(\widetilde{G}(F),K_l^*)\longrightarrow \mathcal{H}_{\varepsilon'}(\widetilde{G}(F'),{K_l'}^*).
\end{align*}
Again, by Lemma \ref{algebra-iso}, we have $ \mathscr{\bar{\tau}}(f_i)=\Phi_l(f_i)$. The elements $f_i$ generate $\mathcal{H}_\varepsilon(\widetilde{G}(F),K_l^*)$, and the double cosets $X(l)$ form a $\mathbb C$-basis for this algebra. Therefore, by Remark~\ref{finitely presented}, $\bar{\tau}=\Phi_l$.  Hence, $\Phi_l$ is an algebra homomorphism.

\begin{corollary} 
Let $F$ and $F'$ be sufficiently close local fields, and let $l$ be as in Theorem~\ref{isomorphism}. Then, we have the following bijection: 
\[
\begin{aligned}
&\{\text{ Irreducible genuine $\bc$-representations } (\pi,V) \text{ of } \gf\text{ such that  } \pi^{K_l^*}\neq 0\} \\
&\longleftrightarrow \{\text{ Irreducible genuine $\bc$-representations } (\pi',V') \text{ of } \hf\text{ such that  } \pi'^{K_l'^*}\neq 0 \} .
\end{aligned}
\]
\end{corollary}

\begin{proof}
By Theorem~\ref{isomorphism}, the genuine Hecke algebras of $\gf$ and $\hf$ are isomorphic. The corollary then follows from Corollary~\ref{equivalence of categories}.
\end{proof}

\section{Equivalence of Categories}
\begin{proposition}\label{properties of congruence}
    The subgroup $K_m^*$ satisfies the following conditions:
    \begin{enumerate}[label=\roman*.]
        \item Let $\tl$ be a Levi subgroup, and let $\tp = \tl U$ be a parabolic subgroup of $\gf$ with Levi $\tl$. Let $K$ be a $\gf$-conjugate of $K_m^*$, and let $K_{\tp}=K\cap \tp/K\cap U$. For any parabolic subgroup $\widetilde{Q}$ of $\gf$ with the same Levi subgroup $\tl$ and any other $\gf$-conjugate $K_0$ of $K_m^*$,  $(K_0)_{\widetilde{Q}}$ is a conjugate of $K_{\tp}$ in $\tl$. 
        \item  For every parabolic $\tp = \tl U$ and every representation $V$ of $\gf$, the canonical map
\[ V^{K_m^*} \longrightarrow (V_U)^{({K_m^*})_{{\tl}}} \] is surjective.
    \end{enumerate}
\end{proposition}

\begin{proof}
Consider the maximal split torus consisting of diagonal matrices $T $ of $G(F)$, and denote its inverse image by $\widetilde{T}$.  It suffices to prove the proposition for the Levi subgroups $\tl$ that contain $\widetilde{T}$. 

Let $\tp$ be a parabolic subgroup containing $\tl$. Then we have the following Iwasawa decomposition
\begin{align}\label{iwasawa}
    \gf=\tp \,G(\mathscr{O})^* .
\end{align}
   Let $\widetilde{\overline{P}}$ be the opposite parabolic of $\tp$, having the same Levi $\tl$. Let $U$ and $\overline{U}$ be the unipotent radicals of $\tp$ and $\widetilde{\overline{P}}$, respectively. We have
   \begin{align*}
       K_m^*= (K_m^*\cap U)(K_m^*\cap \tl)(K_m^*\cap \overline{U}).
   \end{align*}
   
    Hence, $K_m^*\cap \tp=(K_m^*\cap \tl)(K_m^*\cap U)$ and hence $K_m^*\cap \tp/ K_m^* \cap U\simeq K_m^*\cap \tl$.

Now, $K=gK_m^*g^{-1}$, for some $g\in \gf$. From \eqref{iwasawa} and using the fact that $K_m^*$ is normal in $G(\mathscr{O})^*$, we have $K=pK_m^*p^{-1}$, for some $p\in \tp$. Hence
\begin{align*}
    K\cap \tp= p(K_m^*\cap \tp)p^{-1} \qquad \text{and}\qquad K\cap U=p(K_m^*\cap U)p^{-1}.
\end{align*}
Therefore
    \begin{align*}
        K_{\tp}=p\big(K_m^*\cap \tp/K_m^*\cap U)p^{-1}\cong p(K_m^*\cap \tl) p^{-1}.
    \end{align*}
Now, we can write $p=lu$, for some $l\in \tl, u\in U$. Since $U$ acts trivially on the Levi quotient, we have:
\begin{align*}
    K_{\tp}= l (K_m^*\cap \tl) l^{-1}, \quad \text{ for some } l\in\tl.
\end{align*}

    For the proof of \((ii)\), by passing to the inductive limit, we may assume that \(V\) is of finite type. The rest of the proof follows the same argument as in \cite[Prop. 3.5.2]{principal}. 
\end{proof}

\begin{proposition}\label{condition for fixed vectors}

 Let $V$ be a genuine representation of $\gf$. With respect to the notions of supercuspidal pairs and Bernstein blocks introduced in Section~\ref{hecke algebra}, $V$ is generated by its \( K_m^* \)-fixed vectors if and only if the following condition holds:

For every supercuspidal pair \( (\tl, \widetilde{\rho}) \), let \( W \) be a representation of \( \tl \) whose isomorphism class lies in \( \widetilde{\rho} \). If \( W^{(K_m^*)_{\tl}} = 0 \), then the component of \( V \) in  $\mathcal{R}(\gf){(\tl,\widetilde{\rho})}$ is zero.

\end{proposition}
\begin{proof}
    This result follows directly from Proposition~\ref{properties of congruence}, by an argument similar to that of \cite[Prop.~3.8]{BDKV}.
\end{proof}

Let $\mathcal{H}_\varepsilon(\gf)$ be the set of genuine, locally constant, compactly supported functions $f$ on $\gf$ such that $f(\zeta g)=\varepsilon(\zeta)f(g)$ for all $\zeta\in \mu_n(F)$ and $g\in \gf$, which forms an algebra under convolution. Furthermore, let $\mathcal{R}^m(\gf)$ be the subcategory of $\mathcal{R}(\gf)$ consisting of genuine representations $(\sigma, V)$ generated by their $K_m^*$-fixed vectors; that is, $\sigma(\gf)(V^{K_m^*}) = V$.

\begin{corollary}\label{equivalence of categories}
    The category $\mathcal{R}^m(\gf)$ is closed under sub-quotients and hence the functor 
    \begin{align*}
        J_m : \mathcal{R}^m(\gf) &\longrightarrow \mathcal{H}_\varepsilon(\gf, K_m^*)\text{-mod}, \\ (\sigma, V) &\longmapsto V^{K_m^*},
    \end{align*}
    is an equivalence of categories with the left adjoint
    \begin{align*}
        \widetilde{J}_m : \mathcal{H}_\varepsilon(\gf, K_m^*)\text{-mod} &\longrightarrow \mathcal{R}^m(\gf), \\ V^{K_m^*} &\longmapsto \mathcal{H}_\varepsilon(\gf) \otimes_{\mathcal{H}_\varepsilon(\gf, K_m^*)} V^{K_m^*}.
    \end{align*}
\end{corollary}
\begin{proof}
  By the Bernstein decomposition \cite[Theorem 3.9]{KaplanDani}, the category $\mathcal{R}(\gf)$ decomposes as the direct sum
\begin{align*}
    \mathcal{R}(\gf) = \bigoplus_{(\tl,\widetilde{\rho}) \in \mathfrak{B}(\gf)} \mathcal{R}(\gf){(\tl,\widetilde{\rho})},
\end{align*}
where $\mathfrak{B}(\gf)$ denotes the set of equivalence classes of supercuspidal pairs from Section~\ref{hecke algebra}. Let $\Sigma_m \subset \mathfrak{B}(\gf)$ be the subset defined by 
\begin{align*}
    \Sigma_m = \left\{ (\tl, \widetilde{\rho}) \in \mathfrak{B}(\gf) \mid \exists W \in \widetilde{\rho} \text{ such that } W^{(K_m^*)_{\tl}} \neq 0 \right\}.
\end{align*}
It follows from Proposition~\ref{condition for fixed vectors} that the subcategory $\mathcal{R}^m(\gf)$ consisting of representations generated by $V^{K_m^*}$ is given precisely by
\begin{align*}
     \mathcal{R}^m(\gf) = \bigoplus_{(\tl, \widetilde{\rho}) \in \Sigma_m} \mathcal{R}(\gf){(\tl,\widetilde{\rho})}.
\end{align*}
According to \cite[\S 3]{BDKV}, the Bernstein-Zelevinsky finiteness theorem for reductive groups \cite[\S 4]{BernsteinZelevinsky1976} extends naturally to finite central extensions. Hence, for the compact subgroup $K_m^*$, only finitely many inertial classes $(\tl,\widetilde{\rho})$ admit nonzero $K_m^*$-fixed vectors, which implies that the index set $\Sigma_m$ is finite.

\qquad Each Bernstein block $\mathcal{R}(\gf){(\tl,\widetilde{\rho})}$ is a Serre subcategory; hence, it is closed under taking subquotients. Since a finite direct sum of Serre subcategories is again a Serre subcategory, it follows that $\mathcal{R}^m(\gf)$ is also closed under taking subquotients.
\end{proof}

\textbf{Acknowledgements.} I would like to thank my supervisor, Radhika Ganapathy, for suggesting this problem and sharing many valuable insights. I am also grateful to her for her thorough review of the paper and her helpful suggestions on the writing. This work is supported by the NBHM scholarship (Ref. No. 0203/6/2022/R\&D-II/3456) and the DST-FIST programme (grant no. DST FIST-2021 [TPN-700661]).

\bibliographystyle{siam}

\bibliography{ref}

\end{document}